\documentclass[12pt]{article}
\usepackage{fullpage,graphicx,psfrag,url}
\usepackage{epstopdf}
\usepackage{amsmath,amssymb,amsthm,enumitem}
\usepackage{multirow}
\usepackage{tikz}
\usepackage{verbatim,fancyvrb}

\usetikzlibrary{calc}
\fvset{xleftmargin=2em}

\newcommand{\BEAS}{\begin{eqnarray*}}
\newcommand{\EEAS}{\end{eqnarray*}}
\newcommand{\BEQ}{\begin{equation}}
\newcommand{\EEQ}{\end{equation}}
\newcommand{\BIT}{\begin{itemize}}
\newcommand{\EIT}{\end{itemize}}

\newcommand{\eg}{{\it e.g.}}
\newcommand{\ie}{{\it i.e.}}

\newcommand{\ones}{\mathbf 1}
\newcommand{\reals}{{\mbox{\bf R}}}
\newcommand{\integers}{{\mbox{\bf Z}}}
\newcommand{\symm}{{\mbox{\bf S}}}  
\newcommand{\complex}{{\mbox{\bf C}}}

\newcommand{\Rank}{\mathop{\bf Rank}}
\newcommand{\Tr}{\mathop{\bf Tr}}
\newcommand{\diag}{\mathop{\bf diag}}

\newcommand{\Expect}{\mathop{\bf E{}}}

\newcommand{\argmin}{\mathop{\rm argmin}}

\newcommand{\sign}{\mathop{\bf sign}}

\newcounter{oursection}

\newcounter{algorithmctr}[section]
\renewcommand{\thealgorithmctr}{\thesection.\arabic{algorithmctr}}
\newenvironment{algdesc}
   {\refstepcounter{algorithmctr}\begin{list}{}{
       \setlength{\rightmargin}{0\linewidth}
       \setlength{\leftmargin}{.05\linewidth}}
       \rmfamily\small
       \item[]{\setlength{\parskip}{0ex}\hrulefill\par
        \nopagebreak{\bfseries\textsf{Algorithm \thealgorithmctr~}}}}
   {{\setlength{\parskip}{-1ex}\nopagebreak\par\hrulefill} \end{list}}

\newtheorem{theorem}{Theorem}
\newtheorem{lemma}{Lemma}

\title{General Heuristics for Nonconvex Quadratically Constrained Quadratic Programming}
\author{Jaehyun Park \and Stephen Boyd}

\begin{document}
\maketitle

\begin{abstract}
We introduce the \emph{Suggest-and-Improve} framework for general nonconvex
quadratically constrained quadratic programs (QCQPs).
Using this framework, we generalize a number of known methods
and provide heuristics to get approximate
solutions to QCQPs for which no specialized methods are available.
We also introduce an open-source Python
package \texttt{QCQP}, which implements
the heuristics discussed in the paper.
\end{abstract}
\clearpage \tableofcontents \newpage

\section{Introduction}\label{s-intro}
In this paper we introduce the \emph{Suggest-and-Improve} heuristic framework
for general nonconvex quadratically constrained quadratic programs (QCQPs).
This framework
can be applied to general QCQPs for which there are no available specialized methods.
We only briefly mention global methods for solving QCQPs in~\S\ref{s-prev-work},
as the exponential running time of these methods makes them unsuitable for medium- to large-scale problems.
Our main focus, instead, will be on polynomial-time methods for obtaining approximate solutions.

\subsection{Quadratically constrained quadratic programming}
A quadratically constrained quadratic program (QCQP) is an optimization problem that can be written in the following form:
\BEQ\label{qcqp-formulation}
\begin{array}{ll}
\mbox{minimize} & f_0(x) = x^T P_0 x + q_0^T x + r_0 \\
\mbox{subject to} & f_i(x) = x^T P_i x + q_i^T x + r_i \le 0,\quad i=1, \ldots, m,
\end{array}
\EEQ
where $x \in \reals^n$ is the optimization variable, and
$P_i \in \reals^{n \times n}$, $q_i \in \reals^n$, $r_i \in \reals$ are given problem data, for $i = 0, 1, \ldots, m$.
Throughout the paper, we will use $f^\star$ to denote the optimal value of~\eqref{qcqp-formulation}, and $x^\star$ to denote an optimal point, \ie, that attains the objective value of $f^\star$, while satisfying all the constraints.
For simplicity, we assume that all $P_i$ matrices are symmetric, but we do not assume any other conditions such as definiteness. This means that~\eqref{qcqp-formulation}, in general, is a nonconvex optimization problem.

The constraint $f_i(x) \le 0$ is affine if $P_i = 0$.
If we need to handle affine constraints differently from quadratic ones,
we collect all affine constraints and explicitly write them as a single inequality
$Ax \le b$, where $A \in \reals^{p \times n}$ and $b \in \reals^p$ are of
appropriate dimensions, and the inequality $\le$ is elementwise.
Then,~\eqref{qcqp-formulation} is expressed as:
\BEQ\label{qcqp-full}
\begin{array}{ll}
\mbox{minimize} & f_0(x) \\
\mbox{subject to} & f_i(x) \le 0, \quad i=1, \ldots, \tilde{m}
\\ & Ax \le b,
\end{array}
\EEQ
where $\tilde{m} = m - p$ is the number of nonaffine constraints.
When all constraints are affine (\ie, $\tilde{m} = 0$), the problem is a (nonconvex) quadratic program (QP).

Even though we set up~\eqref{qcqp-formulation} in terms of inequality constraints only,
it also allows quadratic equality constraints of the form $h_i(x) = 0$ to be added,
as they can be expressed as two quadratic inequality constraints:
\[
h_i(x) \le 0, \quad -h_i(x) \le 0.
\]
An important example of quadratic equality constraint is $x_i^2 = 1$, which forces $x_i$ to be either $+1$ or $-1$, and thus encoding a Boolean variable.

There are alternative formulations of~\eqref{qcqp-formulation} that are all equivalent.
We start with the \emph{epigraph form} of~\eqref{qcqp-formulation},
which makes the objective function linear (hence convex) without loss of generality,
by introducing an additional scalar variable $t \in \reals$:
\BEQ\label{qcqp-epigraph}
\begin{array}{ll}
\mbox{minimize} & t \\
\mbox{subject to} & f_0(x) \le t \\
& f_i(x) \le 0, \quad i=1, \ldots, m.
\end{array}
\EEQ
It is also possible to write a problem equivalent to~\eqref{qcqp-formulation}
that only has quadratic equality constraints, by introducing an additional variable $s \in \reals^m$:
\[
\begin{array}{ll}
\mbox{minimize} & x^T P_0 x + q_0^T x + r_0 \\
\mbox{subject to} & x^T P_i x + q_i^T x + r_i + s_i^2 = 0,\quad i=1, \ldots, m.
\end{array}
\]

The \emph{homogeneous form} of~\eqref{qcqp-formulation} is a QCQP
with $n+1$ variables and $m+1$ constraints,
where the objective function and lefthand side of every constraint is a quadratic form in the variable,
\ie, there are no linear terms in the variable~\cite{luo2010semidefinite}.
For $i = 0, \ldots, m$, define
\[
\tilde{P}_i = \left[ \begin{array}{cc} P_i & (1/2)q_i \\ (1/2)q_i^T & r_i \end{array} \right].
\]
The homogeneous form of~\eqref{qcqp-formulation} is given by:
\BEQ\label{qcqp-homogeneous}
\begin{array}{ll}
\mbox{minimize} & z^T \tilde{P}_0 z \\
\mbox{subject to} & z^T \tilde{P}_i z \le 0, \quad i=1, \ldots, m \\
& z_{n+1}^2 = 1,
\end{array}
\EEQ
with variable $z \in \reals^{n+1}$.
Note that the variable dimension and the number of constraints of~\eqref{qcqp-homogeneous}
is one larger than that of~\eqref{qcqp-formulation}.
This problem is homogeneous in the sense that scaling
$z$ by a factor of $t \in \reals$ scales both the objective
and lefthand sides of the constraints by a factor of $t^2$.
It is easy to check that if $z^\star$ is a solution of~\eqref{qcqp-homogeneous},
then the vector $(z^\star_1/z^\star_{n+1}, \ldots, z^\star_n/z^\star_{n+1})$ is a solution of~\eqref{qcqp-formulation}.

\subsection{Tractable cases}
The class of problems that can be written in the form of~\eqref{qcqp-formulation} is very broad,
and as we will see in~\S\ref{s-examples}, QCQPs are NP-hard in general.
However, there are a number of special cases which we can solve efficiently.

\paragraph{Convex QCQP.}
When all $P_i$ matrices are positive semidefinite, problem~\eqref{qcqp-formulation} is convex and thus easily solvable in polynomial time~\cite[\S4.4]{boyd2004convex}.

\paragraph{QCQP with one variable.}
If the problem has only one variable, \ie, $n=1$, then the feasible set
is explicitly computable using only elementary algebra.
The feasible set in this case is a collection of at most $m+1$ disjoint, closed intervals on $\reals$.
It is possible to compute these intervals in $O(m \log m)$ time (using a binary search tree, for example).
Then, minimizing a quadratic function in this feasible set can be done by evaluating
the objective at the endpoints of the intervals, as well as checking the
unconstrained minimizer (if there is one).

While one-variable problems are rarely interesting by themselves,
we will take advantage of their solvability in~\S\ref{s-greedy}
to develop a greedy heuristic for~\eqref{qcqp-formulation}.
For more details on the solution method and time complexity analysis, see Appendix~\ref{s-onevar}.

\paragraph{QCQP with one constraint.}
Consider QCQPs with a single constraint (\ie, $m=1$):
\BEQ\label{eq-oneqcqp-general}
\begin{array}{ll}
\mbox{minimize} & f_0(x) \\
\mbox{subject to} & f_1(x) \le 0.
\end{array}
\EEQ
Even when $f_0$ and $f_1$ are both nonconvex,~\eqref{eq-oneqcqp-general} is
solvable in polynomial time~\cite{boyd2004convex,feron2000nonconvex,ye2003new,luo2010semidefinite}.
This result, also known as the $S$-procedure
in control theory~\cite{boyd1994linear,polik2007survey},
states that even though~\eqref{eq-oneqcqp-general}
is not convex, strong duality holds and the Lagrangian relaxation
produces the optimal value $f^\star$.
A variant of~\eqref{eq-oneqcqp-general} with an equality constraint $f_1(x) = 0$ is also efficiently solvable.

In Appendix~\ref{s-oneqcqp}, we derive a solution method for the special case of~\eqref{eq-oneqcqp-general},
where the objective function is given by $f_0(x) = \|x-z\|_2^2$.
This particular form will be used extensively in~\S\ref{s-admm}.
Refer to~\cite{feng2012duality,more1993generalizations,polik2007survey} for the solution methods for the general case.

\paragraph{QCQP with one interval constraint.}
Consider a variant of~\eqref{eq-oneqcqp-general}, with an
interval constraint:
\BEQ\label{eq-oneqcqp-interval}
\begin{array}{ll}
\mbox{minimize} & f_0(x) \\
\mbox{subject to} & l \le f_1(x) \le u.
\end{array}
\EEQ
Solving this variant reduces to solving~\eqref{eq-oneqcqp-general} twice, once
with the upper bound constraint $f_1(x) \le u$ only, and once with the lower
bound constraint $f_1(x) \ge l$ only~\cite{ben1996hidden,huang2016consensus}.
One of the solutions is guaranteed to be an optimal point of~\eqref{eq-oneqcqp-interval}.

\paragraph{QCQP with homogeneous constraints with one negative eigenvalue.}
Consider a homogeneous constraint of the form $x^T P x \le 0$,
where $P$ has exactly one negative eigenvalue.
This constraint can be rewritten as a disjunction of two second-order cone (SOC) constraints~\cite{lobo1998applications}.
Let $P = Q\Lambda Q^T$ be the eigenvalue decomposition of $P$, with $\lambda_1 < 0$. Then, $x^T P x \le 0$ if and only if
\[
\sum_{i=2}^n \lambda_i (q_i^T x)^2 \le -\lambda_1 (q_1^T x)^2,
\]
or equivalently,
\[
\left\|(\sqrt{\lambda_2}q_2^T x, \ldots, \sqrt{\lambda_n}q_n^T x) \right\|_2 \le \sqrt{|\lambda_1|} |q_1^T x|,
\]
where $q_1, \ldots, q_n$ are the columns of $Q$.
Depending on the sign of $q_1^T x$, one of the following SOC inequalities is true if and only if $x^T P x \le 0$:
\begin{eqnarray}
\left\|(\sqrt{\lambda_2}q_2^T x, \ldots, \sqrt{\lambda_n}q_n^T x) \right\|_2 &\le& \sqrt{|\lambda_1|} q_1^T x, \label{soc-rewriting-pos}\\
\left\|(\sqrt{\lambda_2}q_2^T x, \ldots, \sqrt{\lambda_n}q_n^T x) \right\|_2 &\le& -\sqrt{|\lambda_1|} q_1^T x. \label{soc-rewriting-neg}
\end{eqnarray}
Suppose that the constraint $x^T P x \le 0$ was the only nonconvex constraint of~\eqref{qcqp-formulation}.
Then, we can solve two convex problems, one where the nonconvex constraint is replaced with~\eqref{soc-rewriting-pos},
and the other where the same constraint is replaced with~\eqref{soc-rewriting-neg}.
The one that attains a better solution is optimal.
Note that if the sign of $q_1^T x$ of any solution is known,
then that information can be used to avoid solving both~\eqref{soc-rewriting-pos} and~\eqref{soc-rewriting-neg}.
For example, if it is known a priori that some solution $x$ satisfies $q_1^T x \ge 0$,
then only~\eqref{soc-rewriting-pos} needs to be solved.

This approach generalizes to the case where
there are multiple constraints in the form of $x^T P x \le 0$,
where $P$ has exactly one negative eigenvalue.
With $k$ such constraints, one needs to solve $2^k$ convex problems.

\subsection{Algorithm}\label{s-algorithm}

We introduce the \emph{Suggest-and-Improve} framework,
which is a simple but flexible idea that encapsulates all of our heuristics.
The overall algorithm can be summarized in two high-level steps, as shown in Algorithm~\ref{master-algo}.

\begin{algdesc}\label{master-algo}
\emph{Suggest-and-Improve algorithm.}
\begin{tabbing}
1.\ \emph{Suggest.} Find a \emph{candidate point} $x \in \reals^n$. \\
2.\ \emph{Improve.} Run a local method from $x$ to find a point $z \in \reals^n$ that is no worse than $x$. \\
3.\ {\bf return} $z$.
\end{tabbing}
\end{algdesc}
For Algorithm~\ref{master-algo} to be well-defined, we need the notion of better points used in the \emph{Improve} step.
While there are other reasonable ways to define this, we use the following definition throughout the paper.
Let $p_+ = \max\{p, 0\}$ denote the positive part of $p \in \reals$, and
\[
v(x) = \max\{ f_1(x)_+, \ldots, f_m(x)_+ \}
\]
denote the maximum constraint violation of $x \in \reals^n$.
We say $z \in \reals^n$ is \emph{better} than $x \in \reals^n$ if one of the following conditions is met:
\begin{itemize}
	\item Maximum constraint violation of $z$ is smaller than that of $x$, \ie, $v(z) < v(x)$.
	\item Maximum constraint violation of $z$ and $x$ are the same, and the objective function attains smaller value at $z$ than at $x$, \ie, $v(z) = v(x)$ and $f_0(z) < f_0(x)$.
\end{itemize}
In other words, we are defining better points in terms of the lexicographic order of the pair $(v(x), f_0(x))$.
This definition easily extends to the notion of a best point in the set of points
(note that there can be multiple best points in a given set).

The candidate points returned from the \emph{Suggest} method serve as starting points of local methods in the \emph{Improve} step,
and we do not require any condition on them in terms of feasibility.
\emph{Suggest} methods can be randomized, and parallelized to produce multiple candidate points.
\emph{Improve} methods attempt to produce better points (as defined above) than the candidate points.
\emph{Improve} methods can be applied in a flexible manner.
For example, note that composition of any number of \emph{Improve} methods is also an \emph{Improve} method.
One can also apply different \emph{Improve} methods to a single candidate point in parallel.
Similarly, given multiple candidate points, one can apply multiple \emph{Improve}
methods on each candidate point and take a best one.

There are many different ways to implement the \emph{Suggest} and \emph{Improve} methods,
and how they are implemented determines the running time, suboptimality,
and various other properties of the overall algorithm.
Essentially, they can be considered as modules that one can choose from a collection of alternatives.
Throughout the paper, we will explore different options to implement them.

To motivate the discussions, we start by recognizing a wide variety of well-known problem classes and examples that can be formulated as QCQPs, in~\S\ref{s-examples}.
In~\S\ref{s-relax}, we explore ways to implement the \emph{Suggest} step.
Our focus is on the \emph{relaxation} technique,
which is typically used to approximate the solution
to computationally intractable problems by replacing
constraints with some other constraints that are easier to handle.
We will discuss various relaxations of~\eqref{qcqp-formulation} for finding reasonable candidate points,
as well as obtaining a lower bound on the optimal value $f^\star$.
Then, in~\S\ref{s-methods}, we discuss local optimization methods for improving a given candidate point.
We will start with specialized methods for some subclasses of QCQP,
and give three methods that can be applied to general QCQPs.
Finally, we introduce an open-source implementation of these methods in~\S\ref{s-package},
and show several numerical examples in~\S\ref{s-experiment}.

\subsection{Previous work}\label{s-prev-work}

\paragraph{Quadratic programming.}
Research on QP began in the 1950s
(see, \eg,~\cite{frank1956algorithm,hildreth1957quadratic}).
Several hardness results on QP were published
once the concept of NP-completeness and NP-hardness was established in the early 1970s~\cite{cook1971complexity,karp1972reducibility}.
In particular,~\cite{sahni1974computationally} showed that QP with a negative definite quadratic term is NP-hard.
On the other hand, QP with convex objective function was shown to be polynomial-time solvable~\cite{kozlov1980polynomial}.

\paragraph{Quadratically constrained quadratic programming.}
Van de Panne~\cite{van1966programming} studied a special class of QCQPs,
which is to optimize an affine function subject to a single quadratic constraint over a polyhedron.
While problems with a quadratic objective function and multiple quadratic constraints were introduced as early as 1951 in~\cite{kuhn1951nonlinear},
duality results and cutting plane algorithms for solving them were
developed later~\cite{baron1972quadratic}.

\paragraph{QCQP with one constraint.}
Problem~\eqref{eq-oneqcqp-general} arises from many optimization algorithms, and most notably, in trust region methods~\cite{hsia2013revisit,gay1981computing,more1983computing,stern1995indefinite,nocedal1999numerical}.
Eigenvalue or singular value problems are also formulated in this form~\cite{lemon2016low,stern1995indefinite,hsia2013revisit}.
The strong duality result is known under various names in different disciplines.
The term $\mathcal{S}$-procedure is from control theory~\cite{boyd1994linear,polik2007survey}. Variations of the $\mathcal{S}$-procedure are
known in linear algebra in the context of simultaneous diagonalization of symmetric
matrices~\cite{calabi1964linear,uhlig1979recurring}.
Many related results and additional references can be found in~\cite[\S1.8]{horn1991topics}, and~\cite[\S4.10.5]{ben2001lectures}.

\paragraph{Semidefinite programming.}
Semidefinite programming (SDP) and semidefinite programming relaxation (SDR)
are closely related to QCQPs.
The study of SDPs started since the early 1990s~\cite{alizadeh1991combinatorial,nesterov1994interior},
and subsequent research during the 1990s
was driven by various applications, including combinatorial problems~\cite{goemans1995maxcut},
control~\cite{boyd1994linear,scherer1997multiobjective,dullerud2000course},
communications and signal processing~\cite{luo2003applications,davidson2000design,ma2002quasi},
and many other areas of engineering.
For more extensive overviews and bibliographies of SDPs, refer to~\cite{wolkowicz2000handbook,todd2001semidefinite,lewis1996eigenvalue,vandenberghe1996semidefinite}.
The idea of SDR for QCQPs was suggested as early as 1979 in~\cite{lovasz1979shannon},
but the work that started a rapid development of the technique was~\cite{goemans1995maxcut},
which applied SDR to the maximum cut problem and derived a data-independent
approximation factor of $0.87856$.
Since then, SDR has also been applied outside the domain of combinatorial optimization problems~\cite{vandenberghe1996semidefinite,luo2010semidefinite}.

\paragraph{Global methods.}
Global methods for nonconvex problems always find an optimal point and
certify it, but are often slow; the worst-case running time grows exponentially with problem size (unless P is NP).
Many known algorithms for globally solving~\eqref{qcqp-formulation} are based on the \emph{branch-and-bound} framework.
Branch-and-bound generally works by recursively splitting the feasible set
into multiple parts and solving the problem restricted in each of the subdivision,
typically via relaxation techniques.
For more details on the branch-and-bound scheme,
see~\cite{lawler1966branch,balas1968note,moore1991global,boyd2003branch}.
A popular variant of the branch-and-bound scheme is the \emph{branch-and-cut} method,
which incorporates cutting planes~\cite{kelley1960cutting} to tighten the subproblems generated from branching~\cite{padberg1991branch,mitchell2002branch}.
See, for example,~\cite{audet2000branch} for a branch-and-cut method for solving~\eqref{qcqp-formulation}.
Linderoth~\cite{linderoth2005simplicial} proposes an algorithm that partitions the feasible region into the Cartesian product of two-dimensional triangles and rectangles.
Burer and Vandenbussche~\cite{burer2008finite} shows an algorithm that uses SDR as a subroutine.
A variant of the method tailored for nonconvex QPs also exists~\cite{chen2012globally}.

\paragraph{Existing solvers.}
There are a number of off-the-shelf software packages
that can handle various subclasses of QCQP.
We mention some of the solvers here:
GloMIQO~\cite{misener2013glomiqo},
BARON~\cite{sahinidis2014baron},
Ipopt~\cite{waechter2009introduction},
Couenne~\cite{cplex2009v12}, and others
provide global methods for (mixed-integer) QCQPs.
Gurobi~\cite{gurobi},
CPLEX~\cite{cplex2009v12},
MOSEK~\cite{mosek},
and SCIP~\cite{achterberg2009scip}
provide global methods for mixed-integer nonlinear programs, with limited support for nonconvex constraints.
Packages such as
ANTIGONE~\cite{misener2014antigone},
KNITRO~\cite{byrd2006knitro},
and NLopt~\cite{johnson2014nlopt}
provide global and local optimization methods
for nonlinear optimization problems.

\section{Examples and applications}\label{s-examples}
In this section, we show various subclasses of QCQP,
as well as several applications that are more specific.

\subsection{Examples}

\paragraph{Polynomial problems.}
A polynomial optimization problem seeks to minimize a polynomial over a set
defined by polynomial inequalities:
\BEQ\label{poly-opt}
\begin{array}{ll}
\mbox{minimize} & p_0(x) \\
\mbox{subject to} & p_i(x) \le 0, \quad i=1, \ldots, m.
\end{array}
\EEQ
Here, each $p_i: \reals^n \rightarrow \reals$ is a polynomial in $x$. All
polynomial optimization problems can be converted to QCQPs by introducing
additional variables that represent the product of two terms, and appropriate equality
constraints that describe these relations.
For example, in order to represent a
term $x_1^2 x_2$, one can introduce additional variables, say, $u$ and $v$,
and add constraints $u = x_1^2$ and $v = u x_2$.
Then we can simply write $v$ in
place of $x_1^2 x_2$. In general, at most $d-1$ new variables and constraints
are sufficient to describe any term of order $d$. By applying these
transformations iteratively, we can transform the original polynomial
problem into a QCQP with additional variables. As a concrete example, suppose
that we want to solve the following polynomial problem:
\[
\begin{array}{ll}
\mbox{minimize}   & x^3-2xyz+y+2 \\
\mbox{subject to} & x^2+y^2+z^2-1=0,
\end{array}
\]
in the variables $x,y,z\in\reals$. We introduce two new variables
$u,v\in\reals$ along with two equality constraints:
\[
u=x^2,\quad v=yz.
\]
The problem then becomes:
\[
\begin{array}{ll}
\mbox{minimize}   & xu-2xv+y+2 \\
\mbox{subject to} & x^2+y^2+z^2-1=0\\
                  & u-x^2=0\\
                  & v-yz=0,\\
\end{array}
\]
which is now a QCQP in the variables $x,y,z,u,v\in\reals$.

\paragraph{Box-constrained mixed-integer quadratic programming.}
Mixed-integer quadratic programming (MIQP) is the problem of
optimizing a quadratic function over a polyhedron, where some variables
are constrained to be integer-valued.
Typically, MIQP comes with a box constraint that specifies lower and upper bounds on $x$,
in the form of $l \le x \le u$. Formally, it can be written as the following:
\BEQ\label{eq-miqp}
\begin{array}{ll}
\mbox{minimize} & f_0(x) \\
\mbox{subject to} & Ax \le b \\
& l \le x \le u \\
& x_1, \ldots, x_p \in \integers.
\end{array}
\EEQ
We can write the integer constraints as a set of nonconvex quadratic inequalities.
For example, $x_1 \in \{l_1, l_1 + 1, \ldots, u_1\}$ if and only if $l_1 \le x_1 \le u_1$ and
\[
(x_1-k)(x_1-(k+1)) \ge 0,
\]
for all $k = l_1, l_1+1, \ldots, u_1-1$.
By replacing the integer constraints in this way, we can write~\eqref{eq-miqp} in the form of~\eqref{qcqp-formulation}.

\paragraph{Rank-constrained problems.}
Let $X \in \reals^{p \times q}$ be a matrix-valued variable.
The rank constraint $\Rank(X) \le k$
can be written as a quadratic constraint
by introducing auxiliary matrix variables
$U \in \reals^{p \times k}$ and $V \in \reals^{k \times q}$,
and adding an equality constraint $X = UV$.
Note that this is a set of $pq$ equality constraints that are quadratic in the elements of $X$, $U$, and $V$:
\[
x_{ij} = \sum_{r=1}^k u_{ir} v_{rj}, \qquad i = 1, \ldots, p, \quad j = 1, \ldots, q.
\]


\subsection{Applications}\label{s-applications}

\paragraph{Boolean least squares.}
The Boolean least squares problem has the following form:
\BEQ
\label{eq-boolean-ls}
\begin{array}{ll}
\mbox{minimize} & \|Ax-b\|_2^2\\
\mbox{subject to} & x_i  \in \{-1,1\}, \quad i=1, \ldots, n,
\end{array}
\EEQ
in the variable $x\in\reals^n$, where $A \in \reals^{m \times n}$ and $b \in \reals^m$.
This is a basic problem in digital
communications (maximum likelihood estimation for digital
signals).
By writing the Boolean constraint $x_i \in \{-1,1\}$ as $x_i^2 = 1$, we get a QCQP equivalent to~\eqref{eq-boolean-ls}:
\[
\begin{array}{ll}
\mbox{minimize} & x^T A^TA x -2 b^TAx +b^Tb\\
\mbox{subject to} & x_i^2 = 1, \quad i=1, \ldots, n.
\end{array}
\]

\paragraph{Two-way partitioning problems.}
The two-way partitioning problem can be written as the following~\cite[\S5.1.5]{boyd2004convex}:
\BEQ
\label{eq-partitioning}
\begin{array}{ll}
\mbox{maximize}   & x^T W x \\
\mbox{subject to} & x_i^2 = 1, \quad i=1,\ldots,n,
\end{array}
\EEQ
with variable $x \in \reals^n$, where $W\in\reals^{n \times n}$ is symmetric. This problem
is directly a nonconvex QCQP of the form~\eqref{qcqp-formulation}.
(Notice, however, that this is a maximization problem,
equivalent to minimizing the negative of the objective.)
Since the constraints restrict the possible values of each $x_i$ to $+1$ or $-1$,
each feasible $x$ naturally corresponds to the partition
\[
\{ 1, \ldots, n \} = \{ i \;|\; x_i = -1 \} \cup \{ i \;|\;
x_i = +1 \}.
\]
The matrix coefficient $W_{ij}$ can be interpreted as the utility
of having the elements $i$ and $j$ in the same cluster, with
$-W_{ij}$ the utility of having $i$ and $j$ in different clusters.
Then, problem~\eqref{eq-partitioning} can be interpreted
as finding the partition that maximizes the total utility over all pairs of elements.

It is possible to generalize this formulation to $k$-way partitioning problems.
In this variant, we would like to partition the set $\{1, \ldots, n\}$ into $k$ clusters,
where there is utility associated between every pair of elements not belonging to the same cluster:
\[
\begin{array}{ll}
\mbox{maximize}   & \sum_{u=1}^k \sum_{v=1}^k x^{(u)T} W^{(u, v)} x^{(v)} \\
\mbox{subject to} & x^{(u)}_i (x^{(u)}_i - 1) = 0, \quad i=1,\ldots,n, \quad u=1, \ldots, k \\
& \sum_{u=1}^k x^{(u)}_i = 1, \quad i=1,\ldots,n.
\end{array}
\]
Here, the matrices $W^{(u,v)}$ describe utility between clusters $u$ and $v$.
The variables $x^{(u)}$ can be considered as indicator variables that represent which elements belong to cluster $u$.
The first set of equality constraints limit each $x_i^{(u)}$ to be either $0$ or $1$.
The second set of equality constraints states that element $i$ belongs to exactly one of the $k$ clusters.

\paragraph{Maximum cut.}
The maximum cut problem is a classic problem in graph theory and network optimization, and is an instance of two-way partitioning problem~\eqref{eq-partitioning}.
On an $n$-node graph $G = (V, E)$, where the nodes are numbered
from $1$ to $n$, we define weights $W_{ij}$
associated with each edge $(i,j) \in E$. If no edge connects $i$ and $j$, we define $W_{ij}=0$.
The maximum cut problem seeks to find a cut
of the graph with the largest possible weight,
\ie, a partition of
the set of nodes $V$ in two clusters $V_1$ and $V_2$ such that the total
weight of all edges linking these clusters is maximized.
Given an assignment $x \in \{-1, +1\}^n$
of nodes to the clusters, the value of the cut
is defined by
\[
\frac{1}{2}\sum_{i, j: x_i x_j=-1} W_{ij},
\]
which is also equal to
\[
\frac{1}{4}\sum_{i, j} W_{ij}(1-x_ix_j) = -(1/4) x^T W x + (1/4) \ones^T W \ones.
\]
Here, $\ones$ represents a vector with all components equal to one.
The maximum cut problem, then, can be written as:
\BEQ\label{maxcut-statement}
\begin{array}{ll}
\mbox{maximize} & -(1/4) x^T W x + (1/4) \ones^T W \ones \\
\mbox{subject to} & x_i^2 = 1, \quad i = 1, \ldots, n.
\end{array}
\EEQ
We can further rewrite~\eqref{maxcut-statement} so that the objective function is homogeneous in $x$. Let $L$ be the Laplacian matrix of the underlying graph $G$, which is given by
\[
L_{ij} = \left\{
\begin{array}{ll}
-W_{ij} & \mbox{if }i \ne j \\
\sum_{j \ne i} W_{ij} & \mbox{otherwise.}
\end{array}
\right.
\]
If the edge weights $W_{ij}$ are all nonnegative, the Laplacian $L$ is positive semidefinite (see, \eg,~\cite[\S13.1]{godsil2013algebraic}).
Using the Laplacian (and ignoring the constant factor), the maximum cut problem can be written in the same form as~\eqref{eq-partitioning}:
\BEQ\label{maxcut-homogeneous}
\begin{array}{ll}
\mbox{maximize}   & x^T L x \\
\mbox{subject to} & x_i^2 = 1, \quad i=1,\ldots,n.
\end{array}
\EEQ

\label{bisection-intro}
The \emph{maximum graph bisection problem} is a
variant of the maximum cut problem,
which has an additional constraint that the two
clusters $V_1$ and $V_2$ must have the same size, \ie, $\ones^T x = 0$. (In this variant, we assume that $n$ is even.)

\paragraph{Maximum clique.}
The maximum clique problem is to find the complete subgraph of the maximum cardinality in a given graph. The problem can be formulated as a QCQP:
\BEQ\label{eq-maxclique}
\begin{array}{ll}
\mbox{maximize} & \ones^T x \\
\mbox{subject to} & x_i x_j (1 - A_{ij}) = 0, \quad i,j = 1, \ldots, n \\
& x_i (x_i - 1) = 0, \quad i = 1, \ldots, n.
\end{array}
\EEQ
Here, $A \in \{0, 1\}^{n \times n}$ is the adjacency matrix of the given graph, where $A_{ii}$ is defined as $1$ for $i=1, \ldots, n$, for convenience. The first set of constraints can be interpreted as: ``if nodes $i$ and $j$ are both in the clique, \ie, $x_i = x_j = 1$, then they must be connected by an edge, \ie, $A_{ij} = 1$.''

\paragraph{3-satisfiability.}
The 3-satisfiability (3-SAT) problem is an NP-complete problem, which is to find an assignment to a set of Boolean variables $x_1, \ldots, x_n$ that makes a given logical expression true. The given expression is a conjunction of $r$ logical expressions called \emph{clauses}, each of which is a disjuction of three variables, with optional negations. This can be formulated as a quadratically constrained feasibility problem as the following:
\BEQ\label{sat}
\begin{array}{ll}
\mbox{find} & x \\
\mbox{subject to} & Ax + b \ge \ones \\
& x_i (x_i - 1) = 0, \quad i=1, \ldots, n.
\end{array}
\EEQ
Note that this is a feasibility problem, and thus an arbitrary objective function can be optimized (such as $\ones^T x$).
Here, the matrix $A \in \reals^{r \times n}$ and vector $b \in \reals^r$ encode the $r$ clauses in the following way:
\[
A_{ij} = \left\{
\begin{array}{ll}
1 & \mbox{if the }i\mbox{th clause includes }x_i \\
-1 & \mbox{if the }i\mbox{th clause includes the negation of }x_i \\
0 & \mbox{otherwise.}
\end{array}
\right.
\]
Without loss of generality, we assume that no clause contains a variable and its negation at the same time, because such a clause can be left out without changing the problem. Each $b_i$ is set as the number of negated variables in the $i$th clause. For example, the inequality corresponding to a clause $(x_1 \vee \neg x_4 \vee x_6)$ would be
\[
x_1 + (1-x_4) + x_6 \ge 1.
\]


\paragraph{Phase retrieval.}
The phase retrieval problem is to recover a general signal such as an image from the magnitude of its Fourier transform.
There are many application areas of the phase retrieval problem, including, but not limited to: X-ray crystallography, diffraction imaging, optics, astronomical imaging, and microscopy~\cite{elser2003phase,waldspurger2015phase,candes2015phase,candes2015phase2,netrapalli2013phase,qian2015phase,huang2016phase,huang2016convexity,shechtman2015phase}.
While there are different ways to formulate the problem, we give one in the form of a feasibility problem:
\[
\begin{array}{ll}
\mbox{find} & x \\
\mbox{subject to} & (a_i^T x)^2 = b_i, \quad i=1, \ldots, m.
\end{array}
\]
Here, the optimization variable is $x \in \reals^n$, and the problem data are $a_1, \ldots, a_m \in \reals^n$ and $b_1, \ldots, b_m \in \reals$.

\paragraph{Multicast downlink transmit beamforming.}
In the context of communications, the downlink beamforming problem seeks to design a multiple-input multiple-output wireless communication system that minimizes the total power consumption,
while guaranteeing that the users receive certain signal-to-interference noise ratio.
(For more details, refer to~\cite{gershman2010convex,mutapcic2007tractable,sidiropoulos2006transmit,yang2011multiuser,luo2010semidefinite,phan2009spectrum,gopalakrishnan2015high,tran2014conic,sidiropoulos2006transmit}.)
This problem can be formulated as a QCQP:
\BEQ\label{eq-mimo}
\begin{array}{ll}
\mbox{minimize} & \|x\|_2^2 \\
\mbox{subject to} & x^T P_i x \ge 1, \quad i = 1, \ldots, m,
\end{array}
\EEQ
in the variable $x \in \reals^n$, where each $P_i \in \reals^{n \times n}$ is positive semidefinite.

\paragraph{Power system state estimation.}
In power systems, fundamental laws such as Ohm's and Kirchhoff's laws dictate
that all power quantities in power systems are quadratic functions of the voltages.
Many power engineering problems, therefore, are naturally written as QCQPs~\cite{wang2016power,zamzam2016beyond}.
For example, the power system state estimation problem can be formulated as a weighted nonlinear least squares problem:
\[
\begin{array}{ll}
\mbox{minimize} & \sum_{i=1}^m w_i (x^T P_i x - v_i)^2,
\end{array}
\]
in the variable $x \in \reals^n$, where $P_i \in \reals^{n \times n}$ and $v, w \in \reals^m$ are given data.
This is a polynomial optimization problem, which can be reformulated as a QCQP.

\section{Relaxations and bounds}\label{s-relax}

A \emph{relaxation} of an optimization problem is obtained by taking a set of
constraints and replacing it with a different set of constraints, such that
the resulting feasible set contains the feasible set of the original problem.
Relaxations have a property that the optimal value $f^\mathrm{rl}$ gives a
lower bound on the optimal value $f^\star$ of the original problem. Tractable
relaxations are of particular interest, since we can solve them to compute a lower
bound on $f^\star$ of intractable optimization problems. While a solution
$x^\mathrm{rl}$ of a relaxation is generally infeasible in the original
problem, it can still serve as a reasonable starting point of various local
methods (which we discuss in~\S\ref{s-methods}).
But if $x^\mathrm{rl}$ is feasible in the original problem, then
it is also a solution of the original problem.

Our main goal of the section is to implement the \emph{Suggest} method of Algorithm~\ref{master-algo}
via various tractable relaxations of~\eqref{qcqp-formulation}.
As a byproduct, we also get a lower bound on $f^\star$.

\subsection{Spectral relaxation}\label{s-spectral}
First, we explore a relaxation of~\eqref{qcqp-formulation}
that is a generalized eigenvalue problem,
hence the name \emph{spectral relaxation}.
This method generalizes eigenvalue bounds studied in~\cite{delorme1993laplacian,delorme1993performance,poljak1995solving}.
Let $\lambda \in \reals^m_+$ be an arbitrary vector in the nonnegative orthant, and consider the following optimization problem:
\BEQ\label{eq-spectral}
\begin{array}{ll}
\mbox{minimize} & f_0(x) \\
\mbox{subject to} & \sum_{i=1}^m \lambda_i f_i(x) \le 0.
\end{array}
\EEQ
Since $\lambda$ is elementwise nonnegative, every feasible point $x$ of~\eqref{qcqp-formulation} is also feasible in~\eqref{eq-spectral}.
Thus,~\eqref{eq-spectral} is a relaxation of~\eqref{qcqp-formulation},
and its optimal value $f^\mathrm{rl}$ is a lower bound on $f^\star$.
Since~\eqref{eq-spectral} is a QCQP with one constraint~\eqref{eq-oneqcqp-general},
it is a tractable problem that can be solved using methods via matrix pencil,
for example~\cite{feng2012duality,more1993generalizations,polik2007survey}.
It is easy to see that the same idea extends to problems with equality constraints.
Below, we derive and show an explicit bound on $f^\star$ for some examples.

\paragraph{Two-way partitioning problem.}
Let us derive a spectral relaxation of~\eqref{eq-partitioning}.
Note that via relaxation, we are seeking an upper bound on $f^\star$ as opposed to a lower bound, since~\eqref{eq-partitioning} is a maximization problem.
With $\lambda = \ones$, the relaxation is:
\[
\begin{array}{ll}
\mbox{maximize} & x^T W x \\
\mbox{subject to} & \sum_{i=1}^n x_i^2 = \|x\|_2^2 = n.
\end{array}
\]
It is easy to see that a solution $x^\mathrm{rl}$ of this relaxation is given by
\[
x^\mathrm{rl} = \sqrt{n} v,
\]
where $v$ is the eigenvector of unit length
corresponding to the maximum eigenvalue $\lambda_{\max}$ of $W$.
With this value of $x^\mathrm{rl}$, we have $f^\mathrm{rl} = n \lambda_{\max}$,
which is an upper bound on $f^\star$.

We also note that by taking the sign of each entry of $x^\mathrm{rl}$, we get a feasible point $z$ of~\eqref{eq-partitioning}.
Then, evaluating the objective $z^T W z$ gives a lower bound on $f^\star$.
Thus, we get both lower and upper bounds on $f^\star$,
simply by finding the largest eigenvalue and the corresponding eigenvector of $W$.
In~\S\ref{s-methods}, we revisit this idea, and discuss various other methods for finding feasible points of general QCQPs in greater detail.


\paragraph{Multicast downlink transmit beamforming.}
The spectral relaxation of~\eqref{eq-mimo} with $\lambda = \ones$ is:
\[
\begin{array}{ll}
\mbox{minimize} & \|x\|_2^2 \\
\mbox{subject to} & x^T \bar{P} x \ge 1.
\end{array}
\]
Here, $\bar{P} = (1/m) (\sum_{i=1}^m P_i)$.
Let $\lambda_{\max}$ be the largest eigenvalue of $\bar{P}$, and $v$ be the corresponding eigenvector.
Except in the pathological case where $\bar{P} = 0$, we have $\lambda_{\max} > 0$.
By observing that the shortest vector $x^\mathrm{rl}$ satisfying the constraint must be a multiple of $v$, we get
\[
x^\mathrm{rl} = (1/\sqrt{\lambda_{\max}}) v, \quad
f^\mathrm{rl} = 1/\lambda_{\max}.
\]

It is also possible to find a feasible point $z \in \reals^n$ from $x^\mathrm{rl}$.
Let
\[
t = \min_i {(x^\mathrm{rl})^T P_i x^{\mathrm{rl}}},
\]
and set
\[
z = (1/t) x^\mathrm{rl}.
\]
It is easy to check that $\hat{x}$ is a feasible point of~\eqref{eq-mimo}.
By evaluating the objective at this point, we get an upper bound on $f^\star$:
\[
\|z\|_2^2 = (1/t^2) \|x^\mathrm{rl}\|_2^2 = \frac{1}{t^2 \lambda_{\max}}.
\]

\subsection{Lagrangian relaxation}
The Lagrangian relaxation provides another way to find a lower bound on $f^\star$,
and it can be considered as a generalization of the spectral relaxation~\eqref{eq-spectral}.
See~\cite[\S5]{boyd2004convex} or~\cite[\S5]{bertsekas1999nonlinear}
for more background on Lagrangian duality results.
In this section, we derive the Lagrangian relaxation of~\eqref{qcqp-formulation},
which has been studied since the 1980s by Shor and others~\cite{shor1987quadratic}.
To simplify notation, we first define, for $\lambda \in \reals^m_+$,
\[
\tilde{P}(\lambda) = P_0 + \sum_{i=1}^m \lambda_i P_i, \qquad
\tilde{q}(\lambda) = q_0 + \sum_{i=1}^m \lambda_i q_i, \qquad
\tilde{r}(\lambda) = r_0 + \sum_{i=1}^m \lambda_i r_i.
\]
The \emph{Lagrangian} of~\eqref{qcqp-formulation} is given by
\[
\mathcal{L}(x, \lambda) = f_0(x) + \sum_{i=1}^m \lambda_i f_i(x) = x^T \tilde{P}(\lambda) x + \tilde{q}(\lambda)^T x + \tilde{r}(\lambda).
\]
The \emph{Lagrangian dual function} is then
\[
g(\lambda) = \inf_x \mathcal{L}(x, \lambda)
=\left\{
\begin{array}{ll}
\tilde{r}(\lambda) - (1/4)\tilde{q}(\lambda)^T \tilde{P}(\lambda)^\dagger \tilde{q}(\lambda) \quad & \mbox{if }\tilde{P}(\lambda) \succeq 0, \enspace \tilde{q}(\lambda) \in \mathcal{R}(\tilde{P}(\lambda)) \\
-\infty & \mbox{otherwise,}
\end{array}
\right.
\]
where $\tilde{P}(\lambda)^\dagger$ and $\mathcal{R}(\tilde{P}(\lambda))$ denote
the Moore--Penrose pseudoinverse and range of $\tilde{P}(\lambda)$, respectively~\cite[\S A.5]{boyd2004convex}.
Using the Schur complement, we can write the Lagrangian dual problem as a semidefinite program (SDP):
\BEQ\label{sdp-dual}
\begin{array}{ll}
\mbox{maximize} & \alpha \\
\mbox{subject to} & \lambda_i \ge 0, \quad i=1, \ldots, n \\
& \left[ \begin{array}{cc}
\tilde{P}(\lambda) & (1/2)\tilde{q}(\lambda) \\
(1/2)\tilde{q}(\lambda)^T & \tilde{r}(\lambda) - \alpha
\end{array} \right] \succeq 0,
\end{array}
\EEQ
with variables $\lambda \in \reals^n$, $\alpha \in \reals$.

Lagrangian relaxation~\eqref{sdp-dual} and spectral
relaxation~\eqref{eq-spectral} are closely related techniques.
We can show that solving the Lagrangian relaxation~\eqref{sdp-dual} is equivalent to finding the value of
$\lambda$ that achieves the best spectral bound.
This property also implies a natural way of obtaining a candidate point $x^\mathrm{rl}$;
we first solve~\eqref{sdp-dual} to obtain a solution $\lambda^\star$.
Then, using the value of $\lambda^\star$, we solve the spectral relaxation~\eqref{eq-spectral}.
Its solution can be taken as a candidate point $x^\mathrm{rl}$.
\begin{theorem}
Let $d_\lambda$ be the optimal value of~\eqref{eq-spectral} for a given $\lambda \in \reals^m_+$, and $d^\star$ be the optimal value of~\eqref{sdp-dual}. Then,
\[
d^\star = \sup_{\lambda \in \mathbf{R}^m_+} d_\lambda.
\]
\end{theorem}
\begin{proof}
We first use the fact that strong duality holds for~\eqref{eq-spectral}.
The dual problem of~\eqref{eq-spectral} can be derived in the same way we derived~\eqref{sdp-dual}:
\BEQ\label{eq-spectral-dual}
\begin{array}{ll}
\mbox{maximize} & \gamma \\
\mbox{subject to} & \eta \ge 0 \\
& \left[ \begin{array}{cc}
\tilde{P}(\eta \lambda) & (1/2)\tilde{q}(\eta \lambda) \\
(1/2)\tilde{q}(\eta \lambda)^T & \tilde{r}(\eta \lambda) - \gamma
\end{array} \right] \succeq 0,
\end{array}
\EEQ
with variables $\eta, \gamma \in \reals$.
Due to strong duality,~\eqref{eq-spectral-dual} has the same optimal value $d_\lambda$ as~\eqref{eq-spectral}.
Note that taking the supremum over $\lambda$ of $d_\lambda$ for all
possible $\lambda \in \reals^m_+$ is equivalent to
making $\lambda$ in~\eqref{eq-spectral-dual} an additional variable in $\reals^m_+$, and solving the resulting problem.
In other words, $\sup_{\lambda \in \mathbf{R}^m_+} d_\lambda$ is the optimal value of the following optimization problem:
\[
\begin{array}{ll}
\mbox{maximize} & \gamma \\
\mbox{subject to} & \eta \ge 0 \\
& \lambda_i \ge 0, \quad i=1, \ldots, n \\
& \left[ \begin{array}{cc}
\tilde{P}(\eta \lambda) & (1/2)\tilde{q}(\eta \lambda) \\
(1/2)\tilde{q}(\eta \lambda)^T & \tilde{r}(\eta \lambda) - \gamma
\end{array} \right] \succeq 0,
\end{array}
\]
with variables $\eta, \gamma \in \reals$, $\lambda \in \reals^m$.
Since both $\eta$ and $\lambda$ are constrained to be nonnegative,
we can eliminate $\eta$ by replacing $\eta \lambda$ with $\lambda$
and keeping the nonnegativity constraint on $\lambda$.
The resulting problem is then equivalent to~\eqref{sdp-dual}, which proves the identity.
\end{proof}

\subsection{Semidefinite relaxation}\label{s-sdr}
In this section, we derive a semidefinite relaxation (SDR)
using a technique called \emph{lifting}.
This SDR and the Lagrangian dual problem~\eqref{sdp-dual} are Lagrangian duals of each other.
To derive the SDR, we start by introducing a new variable $X = x x^T$ and rewriting the problem as:
\BEQ\label{qcqp-rewriting}
\begin{array}{ll}
\mbox{minimize} & F_0(X, x) = \Tr(P_0 X) + q_0^T x + r_0 \\
\mbox{subject to} & F_i(X, x) = \Tr(P_i X) + q_i^T x + r_i \le 0, \quad i=1, \ldots, m \\
& X = x x^T,
\end{array}
\EEQ
in the variables $X \in \reals^{n \times n}$ and $x \in \reals^n$.
This rewriting is called a \emph{lifting},
as we have embedded the original problem with $n$ variables to a much larger space (of $n^2+n$ dimensions).
By lifting, we obtain an additional property that
the objective and constraints are affine in $X$ and $x$,
except the last constraint $X = xx^T$ which is nonconvex.
When we replace this intractable constraint with $X \succeq x x^T$,
we get a convex relaxation.
By rewriting it using the Schur complement, we obtain an SDR:
\BEQ\label{sdp-relaxation}
\begin{array}{ll}
\mbox{minimize} & \Tr(P_0 X) + q_0^T x + r_0 \\
\mbox{subject to} & \Tr(P_i X) + q_i^T x + r_i \le 0, \quad i=1, \ldots, m \\
& Z(X, x) = \left[ \begin{array}{cc} X & x \\ x^T & 1 \end{array} \right] \succeq 0.
\end{array}
\EEQ
The optimal value $f^\mathrm{sdr}$ of problem~\eqref{sdp-relaxation} is then a lower bound on the optimal value $f^\star$ of~\eqref{qcqp-formulation}.
Under mild assumptions (\eg, feasibility of the original problem),~\eqref{sdp-relaxation} has the same optimal value as~\eqref{sdp-dual}.
If an optimal point $(X^\star, x^\star)$ of~\eqref{sdp-relaxation} satisfies $X^\star = x^\star x^{\star T}$
(or equivalently, the rank of $Z(X^\star, x^\star)$ is one),
then $x^\star$ is a solution of the original problem~\eqref{qcqp-formulation}.


It may look like we have created too much slack by lifting the problem to a higher-dimensional space,
followed by relaxing the equality constraint $X = xx^T$ to a semidefinite cone constraint.
However, we can show that the SDR is tight when the original problem is convex.
\begin{lemma}\label{l-feasible}
Let $(X, x)$ be any feasible point of~\eqref{sdp-relaxation}. Then, $x$ satisfies every convex constraint of~\eqref{qcqp-formulation}.
\end{lemma}
\begin{proof}
It is enough to verify the claim for any arbitrary convex constraint $x^T P x + q^T x + r \le 0$ with $P \succeq 0$. Since $(X, x)$ is a feasible point of~\eqref{sdp-relaxation}, we have $X \succeq xx^T$. Then,
\[
\Tr(PX) \ge \Tr(P xx^T) = x^T P x.
\]
Therefore,
\[
x^T P x + q^T x + r \le \Tr (P X) + q^T x + r.
\]
Thus, if the righthand side is nonpositive, so is the lefthand side. The claim follows.
\end{proof}

\begin{lemma}
If problem~\eqref{qcqp-formulation} is convex, then its SDR is tight, \ie, $f^\star = f^\mathrm{sdr}$.
\end{lemma}
\begin{proof}
Lemma~\ref{l-feasible} implies that if~\eqref{qcqp-formulation} is convex, then for any feasible point $(X, x)$ of~\eqref{sdp-relaxation}, replacing $X$ with $xx^T$ gives another feasible point with lower or equal objective value, \ie, there exists an optimal point of~\eqref{sdp-relaxation} that satisfies $X = xx^T$. Thus, adding an additional constraint $X = xx^T$ to~\eqref{sdp-relaxation},
which makes the problem equivalent to~\eqref{qcqp-rewriting},
does not change the set of solutions, nor the optimal value.
\end{proof}

While the SDR bound is not tight in general, in some cases,
it is possible to give a lower bound on $f^\mathrm{sdr}$ in terms of $f^\star$,
which effectively gives both lower and upper bounds on $f^\star$.
The famous result by~\cite{goemans1995maxcut} guarantees a data-independent approximation factor of $0.87856$ to the maximum cut problem.
Their analysis is based on the fact that the
SDR~\eqref{sdp-relaxation} can be interpreted
as a stochastic version
of~\eqref{qcqp-formulation}~\cite{bertsimas1999semidefinite,luo2010semidefinite}.
This interpretation gives a natural probability distribution to sample points from,
which can be used to implement a randomized \emph{Suggest} method in step 1 of Algorithm~\ref{master-algo}.
There are other problem instances where similar approximation factors can be given.
These approximation factors, however, greatly depend on the underlying problems.
We refer the readers to~\cite{luo2010semidefinite} for a summary of some major results on SDR approximation factors.

\label{page-random}
Let $(X^\star, x^\star)$ be any optimal
solution of~\eqref{sdp-relaxation}. Suppose $x \in \reals^n$ is a Gaussian
random variable with mean $\mu$ and covariance $\Sigma$. Then, $\mu =
x^\star$ and $\Sigma = X^\star - x^\star x^{\star T}$ solve the following
problem of minimizing the expected value of a quadratic form, subject to
quadratic inequalities:
\[
\begin{array}{ll}
\mbox{minimize} & \Expect f_0(x)
= \Tr (P_0 \Sigma) + f_0(\mu) \\
\mbox{subject to} & \Expect f_i(x)
= \Tr (P_i \Sigma) + f_i(\mu)
\le 0, \quad i = 1, \ldots, m \\
& \Sigma \succeq 0,
\end{array}
\]
in variables $\mu \in \reals^n$ and $\Sigma \in \reals^{n\times n}$.
Intuitively, minimizing the expected objective value promotes $\mu$ to move
closer to the minimizer of $f_0$ and $\Sigma$ to be ``small,'' the constraints
counteract this and promote $\Sigma$ to be ``large'' enough that the
constraints hold in expectation.
Since the constraints are satisfied only in expectation, there is no guarantee
that sampling points directly from $\mathcal{N}(\mu, \Sigma)$ gives feasible points
of~\eqref{qcqp-formulation} at all.
In particular, if~\eqref{qcqp-formulation} includes an equality constraint, then it will almost certainly fail.
However, recall that the \emph{Suggest} methods do not require feasibility;
sampling candidate points from $\mathcal{N}(\mu, \Sigma)$ is
a reasonable choice for the \emph{Suggest} method.
These candidate points then serve as good starting points
for the \emph{Improve} methods, which we discuss in detail in~\S\ref{s-methods}.

In certain cases, the probabilistic interpretation also allows us to get a hard
bound on the gap between the optimal values $f^\star$ of the original problem~\eqref{qcqp-formulation}
and $f^\mathrm{sdr}$ of its SDR~\eqref{sdp-relaxation}.
A classic example is that of the maximum cut bound described
in~\cite{goemans1995maxcut}, which states that for an undirected graph with nonnegative edge weights, the SDR of the maximum cut problem attains the following bound:
\[
\alpha f^\mathrm{sdr} \le f^\star \le f^\mathrm{sdr},
\]
for $\alpha \approx 0.87856$. (Note that this is a maximization problem.)

Here, we derive a similar bound
of $\alpha = 2/\pi \approx 0.63661$. This result, also known as Nesterov's $\pi/2$ theorem,
extends the result of~\cite{goemans1995maxcut} to any two-way partitioning problem with $W \succeq 0$.
Note that we do not require the off-diagonal entries of $W$ to be nonpositive, which is necessary to get the stronger bound of $\alpha \approx 0.87856$ in~\cite{goemans1995maxcut}.
\begin{theorem}[Nesterov~\cite{nesterov1998global,nesterov1998semidefinite}]
Let $f^\mathrm{sdr}$ be the optimal value of the SDR of~\eqref{eq-partitioning}, where $W \succeq 0$. Then,
\[
\frac{2}{\pi} f^\mathrm{sdr} \le f^\star \le f^\mathrm{sdr}.
\]
\end{theorem}

\begin{proof}
The SDR of~\eqref{eq-partitioning} is:
\[
\begin{array}{ll}
\mbox{maximize}   & \Tr(W X) \\
\mbox{subject to} & X \succeq 0 \\
                  & X_{ii} = 1, \quad i=1,\ldots,n.
\end{array}
\]
Let $X^\star$ be any solution of the SDR, and $f^\mathrm{sdr} = \Tr(W X^\star)$ be the optimal value of the SDR.
Consider drawing $x$ from the Gaussian distribution $\mathcal{N}(0, X^\star)$, and setting $z = \sign(x)$,
where $\sign(\cdot)$ denotes the elementwise sign function.
Note that $z$ is always a feasible point of~\eqref{eq-partitioning}.
The special form of the objective function allows us to find the expected value of the objective $\Expect (z^T W z)$ analytically:
\[
\Expect(z^T W z)
=\sum_{i,j} W_{ij} \Expect(z_i z_j)
=\frac{2}{\pi}\sum_{i,j} W_{ij} \arcsin(X^\star_{ij})
=\frac{2}{\pi}\Tr(W \arcsin X^\star),
\]
where $\arcsin X^\star$ is a matrix obtained by taking elementwise $\arcsin$ of the entries of $X^\star$.
Since $\arcsin(X^\star)\succeq X^\star$ (see, \eg,~\cite[\S3.4.1.6]{ben2001lectures}) and $W$ is positive semidefinite, we get
\[
\Expect(z^T W z)
= \frac{2}{\pi}\Tr(W \arcsin X^\star)
\ge \frac{2}{\pi}\Tr(W X^\star)
= \frac{2}{\pi} f^\mathrm{sdr}.
\]
On the other hand, since $z$ is always feasible, we have
\[
\Expect(z^T W z) \le f^\star.
\]
Together with the fact that the relaxation attains the optimal value $f^\mathrm{sdr}$ no worse than $f^\star$, \ie, $f^\mathrm{sdr} \ge f^\star$, we have
\[
\frac{2}{\pi} f^\mathrm{sdr} \le \Expect(z^T W z) \le f^\star \le f^\mathrm{sdr}.
\]
\end{proof}

Note that this result not only bounds $f^\mathrm{sdr}$ in terms of $f^\star$, but also
shows an explicit procedure for generating feasible points that have the
expected objective value of at least $(2/\pi) f^\mathrm{sdr}$.
In~\S\ref{s-special-methods}, we revisit this procedure in the context of \emph{Improve} methods.
In practice, sampling
just a few points is enough to obtain a feasible solution that exceeds this
theoretical lower bound.

\subsection{Tightening relaxations}\label{s-valid}
Lower bounds obtained from relaxations can be improved by adding additional quadratic
inequalities to~\eqref{qcqp-formulation} that are satisfied by any solution
of the original problem.
In particular, redundant inequalities that hold for all feasible
points of~\eqref{qcqp-formulation} can still tighten the relaxation.
We note, however, that in order for these inequalities to be useful in practice,
they must be computationally efficient to derive.
For example, the inequality $f_0(x) \le f^\star$ holds for every
optimal point of the problem, but it cannot be added to the problem without knowing the value of $f^\star$.
All the valid inequalities we discuss below, therefore, will be restricted to the ones that can be derived efficiently.

Consider the set of affine inequalities $Ax \le b$ in~\eqref{qcqp-full}. For every vector $x$ satisfying $Ax \le b$, we have
\BEQ\label{shore-ineq}
(Ax - b)(Ax - b)^T = A x x^T A^T - Axb^T - bx^T A + bb^T \ge 0,
\EEQ
where the inequality is elementwise. Each entry of the lefthand side has the form
\[
(a_i^T x - b_i)(a_j^T x - b_j) = x^T a_i a_j^T x - (b_i a_j + b_j a_i)^T x + b_i b_j,
\]
where $a_i \in \reals^n$ is the $i$th row of $A$ (considered as a column vector). These indefinite quadratic inequalities then can be added to the original QCQP~\eqref{qcqp-full} without changing the set of solutions.

In certain special cases, there are other valid inequalities
that can be derived directly from the structure of the feasible set.
For example, take any Boolean problem where the feasible set is given by $\{-1, +1\}^n$.
Since the entries of any feasible $x$ are integer-valued, for any $a \in \integers^n$ and $b \in \integers$, we have
\[
(a^T x - b)(a^T x - (b+1)) \ge 0.
\]
While these are redundant inequalities, adding them to the problem can tighten its relaxations.

This technique can be generalized further;
any exclusive-disjunction of two affine inequalities can be encoded as a quadratic inequality.
Let $a^T x \le b$ and $c^T x \le d$ be two affine inequalities such that for every feasible point of~\eqref{qcqp-formulation},
exactly one of them holds. Then,
\[
(a^T x - b)(c^T x - d) \le 0
\]
is a redundant quadratic inequality that holds for every feasible point $x$.

\subsection{Relaxation of relaxations}

The Lagrangian and semidefinite relaxations~\eqref{sdp-dual}
and~\eqref{sdp-relaxation} are polynomial-time solvable, but in practice, can
be expensive to solve as the dimension of the problem gets larger.
In this section, we explore several ways to further relax the relaxation
methods discussed above to obtain lower bounds on $f^\star$ more efficiently.


We first discuss how to further relax the Lagrangian relaxation~\eqref{sdp-dual}.
Weak duality implies that it is not necessary to solve~\eqref{sdp-dual} optimally in order to obtain a lower bound on $f^\star$;
any feasible point $(\lambda, \alpha)$ of~\eqref{sdp-dual} induces a lower bound on $f^\star$.
We note that $\alpha$ is easy to optimize given $\lambda$.
In fact, when $\lambda$ is fixed, optimizing over $\alpha$
is equivalent to solving the spectral relaxation~\eqref{eq-spectral} with the same value of $\lambda$.


Now, we discuss relaxation methods for the SDR~\eqref{sdp-relaxation}.
Note that the semidefinite constraint $Z(X, x) \succeq 0$ can be written as
an infinite collection of affine constraint $a^T Z(X, x) a \ge 0$ for all $a \in \reals^{n+1}$ of unit length, \ie, $\|a\|_2 = 1$.
For example, if $a$ is the $i$th unit vector, the resulting inequality states that $X_{ii}$ must be nonnegative.
To approximate the optimal value $f^\mathrm{sdr}$ of~\eqref{sdp-relaxation}, one can generate affine inequalities to replace the semidefinite constraint and solve the resulting linear program (LP).
While these affine inequalities can come directly from the valid inequalities we discussed in~\S\ref{s-valid},
it is also possible to adopt a cutting-plane method to generate them incrementally.
\filbreak
\begin{algdesc}\label{alg-cutting}
\emph{Cutting-plane method for solving~(\ref{sdp-relaxation}) via LP relaxation.}
\begin{tabbing}
{\bf given} an optimization problem $\mathcal{P}$ of
the form~(\ref{sdp-relaxation}). \\*[\smallskipamount]
1.\ \emph{Initialize.} Add, for each $i=1, \ldots, m$, constraint $F_i(X, x) \le 0$ to $\mathcal{T}$,
the list of constraints.\\
{\bf repeat} \\
\qquad \= 2.\ Solve~\eqref{sdp-relaxation} with constraints $\mathcal{T}$ and get an optimal point $(X^\star, x^\star)$.\\
\> 3.\ {\bf if} $X^\star \succeq x^\star x^{\star T}$, terminate with solution $(X^\star, x^\star)$.\\
\> 4.\ Otherwise, find $a \in \reals^{n+1}$ such that $a^T Z(X^\star, x^\star) a < 0$.\\
\> 5.\ Add constraint $a^T Z(X, x) a \ge 0$ to $\mathcal{T}$.
\end{tabbing}
\end{algdesc}
In step 4, vector $a \in \reals^{n+1}$ satisfying the condition always exists,
because $Z(X^\star, x^\star) \succeq 0$ is equivalent to $X^\star \succeq x^\star x^{\star T}$.
For example, one can always take $a$ equal to the eigenvector corresponding to any negative eigenvalue of $Z(X^\star, x^\star)$.
It is also possible to adapt the LDL factorization algorithm to verify
whether $Z(X^\star, x^\star)$ is positive semidefinite,
and terminate with a suitable vector $a$ in case it is not.
We note that Algorithm~\ref{alg-cutting}, in general, need not converge, unless additional constraints are met,
\eg, the vector $a$ in step 4 is the eigenvector corresponding to the minimum eigenvalue of $Z(X^\star, x^\star)$.
However, at every iteration of step 2, we get a lower bound on $f^\star$ that is no worse than the value from the previous iteration.
In practice, this means that the algorithm can terminate any time when a good enough lower bound is obtained.

It is also possible to write a second-order cone programming (SOCP) relaxation of~\eqref{sdp-relaxation}, which can be thought of as a middle-point between LP and SDP relaxations~\cite{kim2001second}.
Second-order cone (SOC) constraints are more general than affine constraints, and can encode more sophisticated relations that hold for positive semidefinite matrices.
For example, in order for $Z(X, x)$ to be positive semidefinite, every $2$-by-$2$ principal submatrices of it must be positive semidefinite, \ie, for every pair of indices $i$ and $j$,
\[
\left[ \begin{array}{cc} Z_{ii} & Z_{ij} \\ Z_{ji} & Z_{jj} \end{array} \right] \succeq 0,
\]
or equivalently,
\[
Z_{ij}^2 \le Z_{ii} Z_{jj},
\quad Z_{ii} \ge 0,
\quad Z_{jj} \ge 0.
\]
These three inequalities are also equivalent to the following SOC inequalities~\cite[\S4]{boyd2004convex}:
\BEQ\label{eq-soc-rewriting}
\left\| \left[ \begin{array}{c} 2Z_{ij} \\ Z_{ii}-Z_{jj} \end{array} \right] \right\|_2 \le Z_{ii} + Z_{jj},
\quad Z_{ii} \ge 0,
\quad Z_{jj} \ge 0.
\EEQ
In general, in order for $Z(X, x)$ to be positive semidefinite, we must have, for every $(n+1)$-by-$2$ matrix $A$,
\[
A^T Z(X, x) A \succeq 0,
\]
which can be rewritten as SOC inequalities, just like~\eqref{eq-soc-rewriting}.
From this observation, it is simple to adapt Algorithm~\ref{alg-cutting} to solve~\eqref{sdp-relaxation} via SOCP relaxation.

\section{Local methods}\label{s-methods}
In this section, we implement the \emph{improve} method of Algorithm~\ref{master-algo}
using various local methods.
These methods take an arbitrary point $x \in \reals^n$ that is
not necessarily feasible in~\eqref{qcqp-formulation},
and attempts to find a better point $z \in \reals^n$.
Recall that in~\S\ref{s-intro}, we defined better points in terms of maximum constraint violation and objective value.
In general, finding a feasible point is an NP-hard problem,
for otherwise we can perform bisection on the optimal value of the epigraph form~\eqref{qcqp-epigraph}
and find a solution to arbitrary precision in polynomial time.
For this reason, we do not guarantee convergence of the methods or feasibility of the resulting points, except in some special cases discussed in~\S\ref{s-special-methods}.

\subsection{Special cases}\label{s-special-methods}
We start by investigating some special cases, where we can directly exploit
the problem structure and find a feasible point.
Since we are guaranteed to find a feasible point with these methods,
we also get an upper bound on $f^\star$ by evaluating the objective function
at the resulting point.
As it can be seen from the examples, these heuristics are highly problem dependent.


\paragraph{Partitioning problems.}
The feasible set of the partitioning problem~\eqref{eq-partitioning} is
\[
\mathcal{S} = \{x \,|\, x_1^2 = \cdots = x_n^2 = 1\}.
\]
For any given $x \in \reals^n$, the point $z = \sign(x)$ is always feasible.
It is easy to check that $z$ is a projection of $x$ onto $\mathcal{S}$, \ie, $z$ is the closest feasible point to $x$.

For the maximum graph bisection problem described in page~\pageref{bisection-intro},
we can employ a slightly different method by~\cite{bertsimas1999semidefinite} to satisfy the additional constraint $\ones^T x = 0$.
Given $x \in \reals^n$,
we find a feasible $z$ by setting $z_i = 1$ for indices $i$ corresponding to
the $n/2$ largest entries in $x$, and $z_i = -1$ for the other $n/2$ indices.
(Ties are broken arbitrarily when some entries of $x$ have equal values.)
It can be shown that $z$ is also a projection onto the set of feasible points.

\paragraph{Multicast downlink transmit beamforming.}
The feasible set of~\eqref{eq-mimo} is given by
\[
\mathcal{S} = \{x \,|\, x^T P_1 x \ge 1, \ldots, x^T P_m x \ge 1\}.
\]
Let $x \in \reals^n$ be an arbitrary point. By setting
\[
z = \frac{1}{\min_i x^T P_i x} x,
\]
we get a feasible point $z$~\cite{luo2010semidefinite}.
In other words, $z$ is the smallest multiple of $x$ that makes it
feasible. (Recall that every $P_i$ is positive semidefinite.)
Unlike in the case of partitioning problems,
$z$ is not a projection of $x$ onto $\mathcal{S}$, even when $m = 1$.

\paragraph{Maximum clique.}
Here, we adapt the main idea from the heuristic for the maximum graph bisection problem
and generate a feasible point of the maximum clique
problem~\eqref{eq-maxclique} from a given vector $x \in \reals^n$.

\begin{algdesc}\label{maxclique-round}
\emph{Heuristic for finding a clique from a given vector.}
\begin{tabbing}
{\bf given} $x \in \reals^n$. \\*[\smallskipamount]
1.\ Find a permutation $\pi$ of $\{1, \ldots, n\}$ such that $x_{\pi_1} \ge \cdots \ge x_{\pi_n}$.\\
2.\ \emph{Initialize empty clique.} $C := \emptyset$ and $z := (0, \ldots, 0) \in \reals^n$.\\
{\bf for} $k=1, 2, \ldots, n$ \\
\qquad \= 3.\ {\bf if} $C \cup \{\pi_k\}$ forms a clique, $C := C \cup \{\pi_k\}$ and $z_{\pi_k} := 1$. \\
4.\ {\bf return} $z$.
\end{tabbing}
\end{algdesc}
The intuition behind this heuristic is that $x_i$ with larger value
should be more likely to be included in a clique. In particular,
this heuristic always includes the node $i$ with the highest value of $x_i$.
It is clear that at the end of Algorithm~\ref{maxclique-round}, $C$ is a clique.
In addition, we can show that $C$ is a maximal clique,
\ie, there is no node $k \notin C$ such that $C \cup \{k\}$ is a clique.
\begin{lemma}
The clique generated by Algorithm~\ref{maxclique-round} is maximal.
\end{lemma}
\begin{proof}
Assume on the contrary that the clique $C_1$ generated by
Algorithm~\ref{maxclique-round} is a proper subset of some other clique $C_2$.
Choose an arbitrary $k \in C_2 \setminus C_1$. When $k$ is considered in
Algorithm~\ref{maxclique-round}, the partial clique $C$ maintained by the
algorithm is a subset of $C_1$. Since $C_2$ is a clique and $C_1 \subset C_2$,
$C \cup \{k\}$ is a clique and thus $k$ must have been included in $C$, which
leads to a contradiction.
\end{proof}

\subsection{Coordinate descent}\label{s-greedy}

From this section on, we consider general QCQPs that have no obvious structures that
can be exploited as in~\S\ref{s-special-methods}.
First, we show a \emph{coordinate descent} heuristic for improving a given point.
Coordinate descent is a simple and intuitive method for finding a local minimum of a function.
For more results and references on general coordinate descent algorithms, see~\cite{wright2015coordinate}.

Our greedy descent method is based on the fact that one-variable
QCQPs are tractable.
The algorithm consists of two phases:

\paragraph{Phase I.}
The goal of the first phase is to reduce the maximum constraint violation
and, if possible, reach a feasible point.
Let $x \in \reals^n$ be a given candidate point.
We repeatedly cycle over each coordinate $x_j$ of $x$,
and update it to the value that minimizes the maximum constraint violation.
In other words, at each step, we solve:
\BEQ\label{eq-phaseone}
\begin{array}{ll}
\mbox{minimize} & s \\
\mbox{subject to} & f_i(x) \le s, \quad i=1, \ldots, m,
\end{array}
\EEQ
with variables $x_j, s \in \reals$.
For a fixed value of $s$, it is easy to adapt the method of Appendix~\ref{s-onevar}
to check if~\eqref{eq-phaseone} is feasible in $x_j$.
Therefore, by performing bisection on $s$,
we can find the minimum possible value of $s$ to arbitrary precision,
as well as a value of $x_j$ that attains the maximum violation of $s$.
We note that in order to apply the method of Appendix~\ref{s-onevar},
we need to extract the quadratic, linear, and constant coefficients of each $f_i$ in $x_j$.
If the $P_i$ matrices are not sparse, then
evaluating these coefficients can dominate the running time.

When the optimal value of $s$ is zero or smaller,
\ie, a feasible point is found, then phase I ends and phase II begins.
On the other hand, if the maximum constraint violation cannot be improved for any $x_j$,
then phase I terminates unsuccessfully (and needs a new candidate point).

\paragraph{Phase II.}
Phase II starts once a feasible point is found. In this phase, we restrict ourselves to feasible points only, and look for another feasible point with strictly better objective value.
Again, we cycle over each coordinate $x_j$ of $x$ and optimize the objective function while maintaining feasibility.
In other words, we solve~\eqref{qcqp-formulation} with all variables fixed but $x_j$.
The implementation of phase II is a direct application of the method in Appendix~\ref{s-onevar}.
While it is possible to run phase II until no improving direction is remaining,
it can terminate at any point and will still yield a feasible point,
as well as an upper bound on $f^\star$.

Phase II of the coordinate descent method generalizes local search
methods for many combinatorial problems, such as the 1-opt local search heuristic, which have been studied since the 1950s~\cite{croes1958method}.

\subsection{Convex-concave procedure} \label{s-dccp}

The convex-concave procedure (CCP) is a powerful heuristic method for finding a local optimum of difference-of-convex (DC) programming problems, which have the following form:
\BEQ\label{eq-dc-problem}
\begin{array}{ll}
\mbox{minimize} & f_0(x) - g_0(x) \\
\mbox{subject to} & f_i(x) - g_i(x) \le 0, \quad i=1, \ldots, m,
\end{array}
\EEQ
where $f_i : \reals^n \rightarrow \reals^n$ and
$g_i : \reals^n \rightarrow \reals^n$ for $i = 0, \ldots, m$ are convex.
We refer the readers to~\cite{lipp2016variations} for extensive review and bibliography of CCP.

The main motivation for considering CCP to solve QCQPs is that
any quadratic function can be easily rewritten as a DC expression.
Consider, for example, an indefinite quadratic expression: $x^T P x + q^T x + r$.
We decompose the matrix $P$ into
the difference of two positive semidefinite matrices:
\BEQ\label{quadratic-split}
P=P_+ - P_-,\qquad P_+, \ P_- \succeq 0.
\EEQ
Such decomposition is always possible, by taking, for example, $P_+ = P + tI$ and $P_- = tI$ for large enough $t > 0$.
Then, we can explicitly rewrite the expression as the difference of two convex quadratic expressions:
\[
(x^T P_+ x + q_0^T x + r_0) - x^T P_- x.
\]

Once~\eqref{qcqp-formulation} is rewritten as a DC problem,
any CCP method for locally solving DC problems can be applied.
Here, we consider the penalty CCP method in~\cite{lipp2016variations},
which does not require the initial point to be feasible.
The penalty CCP method seeks to optimize the convexified version of the objective,
with an additional penalty on the convexified constraint violations.
The method gradually increases the penalty on constraint violation over iterations.
Convexification of the functions is done by linearizing each $g_i$ around the
current iterate $x^k$:
\BEQ\label{eq-linearize}
\hat{g}_i(x; x^k) = g_i(x^k) + \nabla g_i(x^k)^T (x-x^k).
\EEQ
Formally, the penalty CCP method can be written as below.
\begin{algdesc}\label{penalty-ccp}
\emph{Penalty CCP.}
\begin{tabbing}
{\bf given} initial point $x^0$, penalty parameters $\tau_0 > 0$, $\tau_\mathrm{max}>0$, and $\mu > 1$.\\*[\smallskipamount]
{\bf for} $k=0, 1, \ldots$ \\
\qquad \= 1.\ \emph{Convexify.} According to~\eqref{eq-linearize}, form
$\hat{g}_i(x;x^k)$ for $i = 0, \ldots, m$. \\
\> 2.\ \emph{Solve.} Set the value of $x^{k+1}$ to a solution of\\
\> \qquad $\begin{array}{ll}
\mbox{minimize} & f_0(x)- \hat{g}_0(x;x^k) + \tau_k\sum_{i=1}^m s_i\\
\mbox{subject to}& f_i(x) - \hat{g}_i(x;x^k) \leq s_i,\quad i=1,\ldots,m\\
&s_i \geq 0,\quad i = 1, \ldots, m.
\end{array}$\\
\> 3.\ \emph{Increase constraint violation penalty.} $\tau_{k+1} := \min(\mu\tau_k,\tau_\mathrm{max})$.\\
{\bf until} stopping criterion is satisfied.
\end{tabbing}
\end{algdesc}
There are a number of reasonable stopping criteria~\cite{lipp2016variations}.
For example, Algorithm~\ref{penalty-ccp} can terminate when a feasible point is found,
or the maximum penalty parameter is reached, \ie, $\tau_k = \tau_{\max}$.
When implementing Algorithm~\ref{penalty-ccp},
there are other factors to consider than
the initial point or the penalty parameters;
the performance of the algorithm can vary
depending on how indefinite quadratic
functions are split into the difference of
two convex quadratic expressions.
In Appendix~\ref{s-split}, we discuss various ways
of splitting quadratic functions into convex
parts.


\subsection{Alternating directions method of multipliers}\label{s-admm}
The alternating directions method of multipliers (ADMM) is
an operator splitting algorithm that is originally devised to solve
convex optimization problems~\cite{boyd2011distributed}.
However, due to the flexibility of the ADMM framework, it has been explored as a heuristic to solve nonconvex problems. (See, \eg,~\cite[\S9]{boyd2011distributed}, or~\cite{derbinsky2013improved,huang2016consensus}.)
Here, we consider an adaptation of the algorithm to the QCQP, as considered in~\cite{huang2016consensus}.
Note that due to nonconvexity of QCQPs, typical convergence results on ADMM do not apply.

Let $\mathcal{C} \subseteq \reals^n$ be a given set,
and consider the following variant of~\eqref{qcqp-formulation}:
\BEQ\label{qcqp-full2}
\begin{array}{ll}
\mbox{minimize} & f_0(x) \\
\mbox{subject to} & f_i(x) \le 0, \quad i=1, \ldots, m \\
& x \in \mathcal{C}.
\end{array}
\EEQ
With $\mathcal{C} = \reals^n$, problem~\eqref{qcqp-full2}
is the same as~\eqref{qcqp-formulation}.
In this section, we use~\eqref{qcqp-full2} as our primary problem formulation
rather than~\eqref{qcqp-formulation},
and handle the last constraint $x \in \mathcal{C}$
differently from the other constraints.

To apply ADMM, we take~\eqref{qcqp-full2} and form an equivalent
problem with a consensus constraint:
\[
\begin{array}{ll}
\mbox{minimize} & f_0(z) + \mathcal{I}_\mathcal{C} (z)
+ \sum_{i=1}^m \mathcal{I}_i(x_i) \\
\mbox{subject to} & z = x_1 = \cdots = x_m,
\end{array}
\]
with variables $z, x_1, x_2, \ldots, x_m \in \reals^n$.
It is important to note that $x_i$ does not represent
the $i$th component of $x$; rather,
it represents the $i$th copy of the variable $x$,
all of which should be equal to each other.
The function $\mathcal{I}_\mathcal{C}$ is the $0$--$\infty$ indicator of the set $\mathcal{C}$:
\[
\mathcal{I}_\mathcal{C}(z) =
\left\{
\begin{array}{ll}
0 & z \in \mathcal{C} \\
\infty & z \notin \mathcal{C}.
\end{array}
\right.
\]
Similarly, $\mathcal{I}_i$ is the $0$--$\infty$ indicator function of the constraint $f_i(x) \le 0$:
\[
\mathcal{I}_i(x) =
\left\{
\begin{array}{ll}
0 & f_i(x) \le 0 \\
\infty & f_i(x) > 0.
\end{array}
\right.
\]
The augmented Lagrangian of the problem is:
\[
\mathcal{L}_\rho (z, x_1, \ldots, x_m, u_1, \ldots, u_m) = f_0(z) + \mathcal{I}_\mathcal{C} (z) + \sum_{i=1}^m \left( \mathcal{I}_i (x_i) + \rho \left(\|z-x_i+u_i\|_2^2 - \|u_i\|_2^2 \right) \right),
\]
where each $u_i \in \reals^n$ are \emph{scaled dual variables}~\cite[\S3.1.1]{boyd2011distributed}.
The \emph{penalty parameter} $\rho > 0$ controls the convergence behavior or ADMM.
We postpone the discussion on how to choose the parameter,
and describe the ADMM iteration first:
\BEAS
z^{k+1} &=& \argmin_z \mathcal{L}_\rho (z, x_1^k, \ldots, x_m^k, u_1^k, \ldots, u_m^k),\\
x_i^{k+1} &=& \argmin_{x_i} \mathcal{L}_\rho (z^{k+1}, \ldots, x_{i-1}^{k+1}, x_i, x_{i+1}^k, \ldots, x_m^k, u_1^k, \ldots, u_m^k), \quad i=1, \ldots, m,\\
u_i^{k+1} &=& u_i^k + z^{k+1} - x^{k+1}_i, \quad i=1, \ldots, m.
\EEAS

The update rules for $z$ and $x$ each involves solving an optimization problem,
but some observations allow us to simplify them.
To simplify the $z$-update, we ignore the terms in the augmented Lagrangian that do not depend on $z$.
Then, $z^{k+1}$ is given by the solution of the following QCQP:
\BEQ\label{admm-z-update}
\begin{array}{ll}
\mbox{minimize} & f_0(z) + \rho \sum_{i=1}^m \|z-(x^k_i-u^k_i)\|_2^2 \\
\mbox{subject to} & z \in \mathcal{C}.
\end{array}
\EEQ
in the variable $z \in \reals^n$.
There are a number of cases where~\eqref{admm-z-update} is tractable.
\begin{itemize}
\item
When the penalty parameter $\rho > 0$ satisfies
\[
\lambda_{\min} + m \rho \ge 0,
\]
where $\lambda_{\min}$ is the smallest eigenvalue of $P_0$,
then~\eqref{admm-z-update} is a convex problem and thus is tractable.
\item
If the constraint $z \in \mathcal{C}$ can be written as a single quadratic inequality constraint,
then~\eqref{admm-z-update} is a QCQP with a single constraint~\eqref{eq-oneqcqp-general},
which is tractable regardless of convexity of $\mathcal{C}$ or the value of $\rho$.
\item
When $\mathcal{C} = \reals^n$ and the penalty parameter $\rho$ is chosen such that the objective function is strictly convex, \ie,
\[
\lambda_{\min} + m \rho > 0,
\]
then we can perform the $z$-update by solving a single linear system.
To see this, rewrite the objective function of~\eqref{admm-z-update} as
\[
z^T \tilde{P} z + 2 \tilde{q}^T z + \tilde{r},
\]
with appropriate $\tilde{P} \in \reals^{n \times n}$, $\tilde{q} \in \reals^n$, and $\tilde{r} \in \reals$.
Since $\tilde{P} \succ 0$, the minimizer of the expression above can be found
by simply setting the gradient with respect to $z$ equal to zero:
\[
2\tilde{P} z + 2\tilde{q} = 0,
\]
or equivalently,
\[
\tilde{P}z = -\tilde{q}.
\]
Assuming that the value of $\rho$ is fixed for every iteration,
the coefficient matrix $\tilde{P}$ on the lefthand side stays the same for every $z$-update.
Using this observation, we can perform the $z$-update
efficiently as follows: we compute the factorization of the coefficient matrix $\tilde{P}$ once,
and for every subsequent $z$-update, use the cached factorization to carry out the back-solve.
The improvement in running time is the most significant when $\tilde{P}$ is dense,
as the cost of factorization is $O(n^3)$,
and each back-solve only takes $O(n^2)$ time.
\end{itemize}

The $x$-updates can be simplified as well by ignoring terms that do not depend on $x_i$. Then, $x_i^{k+1}$ is given by the solution of the following QCQP:
\BEQ\label{admm-x-update}
\begin{array}{ll}
\mbox{minimize} & \|x_i - (z^{k+1}+u^k_i) \|_2^2 \\
\mbox{subject to} & f_i(x_i) \le 0,
\end{array}
\EEQ
with variable $x_i \in \reals^n$.
Note that~\eqref{admm-x-update} is a special form
of~\eqref{eq-oneqcqp-general}, which is tractable.
We cover the solution methods for solving~\eqref{admm-x-update} in Appendix~\ref{s-oneqcqp}.
Since~\eqref{admm-x-update} only depends on $z^{k+1}$ and $u^k_i$, all $x$-updates trivially parallelizes over $i = 1, \ldots, m$.

Below, we discuss several extensions of the ADMM-based \emph{Improve} method.

\paragraph{Equality constraints.}
When some constraints are equality constraints,
then only the corresponding $x$-updates need to be modified accordingly.
For example, if the $i$th constraint is an equality constraint $f_i(x) = 0$,
then in order to perform the $x$-update, we would solve
\[
\begin{array}{ll}
\mbox{minimize} & \|x_i - (z^{k+1}+u^k_i)\|_2^2 \\
\mbox{subject to} & f_i(x_i) = 0,
\end{array}
\]
instead of~\eqref{admm-x-update}.

\paragraph{Convex constraints.}
Note that any number of convex constraints of~\eqref{qcqp-formulation} can be
encoded in the constraint $x \in \mathcal{C}$.
In general, this will change the behavior of the algorithm, including how quickly the $z$-update~\eqref{admm-z-update} can be performed.

\paragraph{Two-phase ADMM.}
While the ADMM update rules take a simple form and are easy to implement,
it is still a heuristic applied to a nonconvex problem,
and thus in practice, it may be more important to set
the initial point and the penalty parameter $\rho$
carefully.
Here, we show an adaptation of the two-phase ADMM method by Huang and Sidiropoulos~\cite{huang2016consensus},
that can attain faster convergence in practice.
In phase I, much like the two-phase coordinate descent algorithm introduced in~\S\ref{s-greedy},
the algorithm first focuses on finding a feasible point by ignoring the objective function.
Once a feasible point is found, phase II begins. In phase II, the objective function is
brought back into consideration and ADMM iterations
are performed until convergence.

The only difference of the two phases is that in phase I,
the objective function $f_0$ is completely ignored.
This simplifies the $z$-update of phase I; in order to perform the $z$-update,
we solve the following optimization problem:
\BEQ\label{admm-z-update-feasibility}
\begin{array}{ll}
\mbox{minimize} & \sum_{i=1}^m \|z-(x^k_i-u^k_i)\|_2^2 \\
\mbox{subject to} & z \in \mathcal{C},
\end{array}
\EEQ
with variable $z \in \reals^n$.
Define
\[
\bar{z} = \frac{1}{m}\sum_{i=1}^m (x^k_i-u^k_i).
\]
The solution of~\eqref{admm-z-update-feasibility} is simply
given by the projection of $\bar{z}$ onto $\mathcal{C}$.
As long as $\mathcal{C}$ is convex, the projection can be found efficiently.
The $x$- and $u$-update rules stay the same.
Notice that the new $z$-update rule is independent of $\rho$.
In other words, the ADMM iterates in phase I
are completely determined by the initial point.

\label{admm-loop}
We note that depending on the initialization,
the iterates in phase I can be stuck in an infinite loop
of period greater than 1, even for very small problems.
Consider a $3$-dimensional two-way partitioning problem~\eqref{eq-partitioning}, which has three equality constraints
\[
x_i^2 = 1, \quad i=1, 2, 3.
\]
Consider the following initial points:
\BEAS
z^0   &=& (1/3) \ones, \\
x_1^0 &=& (-1, 1/3, 1/3), \\
x_2^0 &=& (1/3, -1, 1/3), \\
x_3^0 &=& (1/3, 1/3, -1), \\
u_1^0 &=& (2/3, 0, 0), \\
u_2^0 &=& (0, 2/3, 0), \\
u_3^0 &=& (0, 0, 2/3).
\EEAS
For this particular initialization, we get
\[
z^{k+1} = -z^k, \quad x^{k+1}_i = -x^k_i, \quad u^{k+1}_i = -u^k_i,
\]
and thus the iterates repeat themselves with period $2$.
It can be verified that any initialization close to this will also fail to converge to a feasible point.

\section{Implementation}\label{s-package}
We introduce an open source Python package \texttt{QCQP} that accepts
high-level description QCQPs as input and implements Algorithm~\ref{master-algo}.
As our main platform, we used \texttt{CVXPY}, a domain-specific language for convex optimization~\cite{cvxpy}.
The source code repository for \texttt{QCQP} is available at \url{https://github.com/cvxgrp/qcqp}.

\subsection{Quadratic expressions}
The \texttt{CVXPY} parser determines curvature, sign, and monotonicity of expressions according to the \emph{disciplined convex programming} (DCP) rules~\cite{grant2006disciplined}.
We have extended the parser so that it can also determine quadraticity of an expression, based on the rules that are similar to the DCP rules.
The following list of expressions is directly recognized as quadratic by \texttt{CVXPY}:
\begin{itemize}
	\item any constant or affine expression
	\item any affine transformation of a quadratic expression, \eg, the sum of quadratic expressions
	\item product of two affine expressions
	\item elementwise square \verb|power(X, 2)| or \verb|square(X)|, with affine \verb|X|
	\item \verb|sum_squares(X)| with affine \verb|X|, representing $\sum_{ij} X_{ij}^2$.
	\item \verb|quad_over_lin(X, c)| with affine \verb|X| and positive constant \verb|c|, representing $(1/c) \sum_{ij} X_{ij}^2$
	\item \verb|matrix_frac(x, P)| with affine \verb|x| and symmetric constant \verb|P|, representing $x^T P^{-1} x$
	\item \verb|quad_form(x, P)| with affine \verb|x| and symmetric constant \verb|P|, representing $x^T P x$
\end{itemize}

\subsection{Constructing problem and applying heuristics}\label{s-constructing}
Our implementation can handle QCQPs constructed using the standard \texttt{CVXPY} syntax.
As long as the objective function and both sides of the constraints are quadratic, the problem is accepted even when it is nonconvex.
In order to apply the \emph{Suggest-and-Improve} framework,
a \texttt{CVXPY} problem object must be passed to the \texttt{QCQP} constructor first.
For example, if \verb,problem, is a \texttt{CVXPY} problem object, then the following code checks whether \verb,problem, describes a QCQP, and if so, prepares the \emph{Suggest} and \emph{Improve} methods:
\begin{Verbatim}
qcqp = QCQP(problem)
\end{Verbatim}
Once the \verb,qcqp, object is constructed, a number of different \emph{Suggest} and \emph{Improve} methods can be invoked on it.
Currently, three \emph{Suggest} methods are available:
\begin{itemize}
\item \verb,qcqp.suggest(), or \verb,qcqp.suggest(RANDOM), fills the values
of the variables using independent and identically distributed Gaussian random variables.
\item \verb,qcqp.suggest(SPECTRAL), fills the values of the variables with a solution of the spectral relaxation~\eqref{eq-spectral}.
The spectral lower bound (or upper bound, in the case of a maximization problem) on the optimal value $f^\star$ is accessible via \verb,qcqp.spectral_bound,.
\item \verb,qcqp.suggest(SDR), solves the SDR~\eqref{sdp-relaxation} and
fills the values of the variables according to the probabilistic interpretation discussed in page~\pageref{page-random}.
The SDR bound on the optimal value is accessible via \verb,qcqp.sdr_bound,.
\end{itemize}
Below is a list of available \emph{Improve} methods:
\begin{itemize}
\item \verb,qcqp.improve(COORD_DESCENT), performs the two-stage coordinate descent method, as described in~\S\ref{s-greedy}.
\item \verb,qcqp.improve(DCCP), rewrites the problem in the DC form~\eqref{eq-dc-problem} then runs the penalty CCP method in~\S\ref{s-dccp} using the open source Python package \texttt{DCCP}.
\item \verb,qcqp.improve(ADMM), runs the two-phase ADMM, as described in~\S\ref{s-admm}.
\end{itemize}
As mentioned in~\S\ref{s-algorithm}, composition of any number of \emph{Improve} methods is also an \emph{Improve} method.
It is easy to apply a sequence of \emph{Improve} methods by passing a list of methods to \verb|improve()|:
\begin{Verbatim}
qcqp.improve(method_sequence)
\end{Verbatim}
This is equivalent to:
\begin{Verbatim}
for method in method_sequence:
    qcqp.improve(method)
\end{Verbatim}
Various parameters can be supplied to the \emph{Suggest} and \emph{Improve} method,
such as the penalty parameter $\rho$ of the two-phase ADMM,
penalty parameter $\tau$ of the penalty CCP,
maximum number of iterations,
and tolerance value for determining near-zero quantities.
All \emph{Suggest} and \emph{Improve} methods return a pair $(f_0(x), v(x))$,
\ie, the objective value and maximum constraint violation at the current point $x$.

\subsection{Sample usage}\label{s-usage}
In this section, we show a sample usage of the \texttt{QCQP} package with a small two-way partitioning problem~\eqref{eq-partitioning} with $n = 10$.
We start by importing the necessary packages, and constructing a symmetric matrix $W \in \reals^{n \times n}$:
\begin{Verbatim}
import numpy as np, cvxpy as cvx
from qcqp import *

n = 10
W0 = np.random.randn(n, n)
W = 0.5*(W0 + W0.T)
\end{Verbatim}
For clarity, we imported \texttt{NumPy} and \texttt{CVXPY} with explicit namespaces.
Next, we construct a \texttt{CVXPY} problem instance describing~\eqref{eq-partitioning}.
\begin{Verbatim}
x = cvx.Variable(n)
prob = cvx.Problem(
    cvx.Maximize(cvx.quad_form(x, W)),
    [cvx.square(x) == 1]
)
\end{Verbatim}
While \texttt{CVXPY} allows defining \verb|prob|, it is not possible to invoke the \verb|solve()| method on it because
the objective function is not concave, and the constraints are nonconvex.
However, we can pass it as an argument to the \texttt{QCQP} constructor to indicate that \verb|prob| is a QCQP:
\begin{Verbatim}
qcqp = QCQP(prob)
\end{Verbatim}
Now we can call \verb|suggest()| and \verb|improve()| methods on \verb|qcqp|.
Here, we solve the spectral relaxation~\eqref{eq-spectral} with $\lambda = \ones$,
for which both the optimal value and solution are known.
Although it is not necessary, we use the MOSEK solver~\cite{mosek} for robustness.
\begin{Verbatim}
qcqp.suggest(SPECTRAL, solver=cvx.MOSEK)
\end{Verbatim}
This fills \verb|x.value|, the numerical value of \verb|x|, with the solution of the spectral relaxation.
The spectral bound can be accessible via \verb|qcqp.spectral_bound|.
We print out this value and compare it with $f^\mathrm{rl} = n \lambda_{\max}$ shown in~\S\ref{s-spectral}.
\begin{Verbatim}
print ("Spectral bound: %.4f" % qcqp.spectral_bound)

(w, v) = np.linalg.eig(W)
print ("n*lambda_max:   %.4f" % (max(w)*n))
\end{Verbatim}
Indeed, we see that both values coincide:
\begin{Verbatim}
Spectral bound: 31.2954
n*lambda_max:   31.2954
\end{Verbatim}
We can also verify that the value \verb|x.value| is correctly populated
with the (scaled) eigenvector of $W$ corresponding to its maximum eigenvalue.
Next, we apply the two-phase coordinate descent heuristic on \verb|x.value|,
and print out the objective value and maximum constraint violation at the resulting point.
\begin{Verbatim}
f_cd, v_cd = qcqp.improve(COORD_DESCENT)
print ("Objective: %.4f" % f_cd)
print ("Maximum violation: %.4f" % v_cd)
\end{Verbatim}
In this example, we get a feasible point:
\begin{Verbatim}
Objective: 23.1687
Maximum violation: 0.0000
\end{Verbatim}
We can print out \verb|x.value| and check that every coordinate is indeed $\pm 1$.
Since $n$ is small, we can enumerate every one of $2^n= 1024$ feasible points
and verify that \verb|x.value| is a global solution of the problem.
See Appendix~\ref{s-source-codes} for the full version of the script.

\section{Numerical examples}\label{s-experiment}
In this section, we consider two numerical examples and
perform the heuristics implemented in the \texttt{QCQP} package.
In Appendix~\ref{s-source-codes}, we give sample Python script for
writing these example problems as QCQPs and applying the
\emph{Suggest-and-Improve} framework using our package.
More examples can be found in the package repository.

\paragraph{Specialized methods.}
We note that any kind of method tailored to particular problems,
\eg, maximum cut~\eqref{maxcut-statement},
multicast beamforming~\eqref{eq-mimo}, or other well-known problems described in~\S\ref{s-examples},
will almost certainly perform better than \texttt{QCQP}, and our objective is not to compete with these specialized methods.
The primary goal of the package, instead, is to provide an easily accessible interface
to various heuristics for NP-hard QCQPs that do not have specialized methods.

\paragraph{Running time.}
Currently, \texttt{QCQP} supports minimal parallelism and
introduces computational overhead by explicitly representing
quadratic expressions by their coefficient matrices.
The implementation can be improved further by better exploiting parallelism,
and rewriting parts of the codes in low-level languages such as C.
Further optimizing the performance of the heuristics is left as future work.
The reported running times are CPU times (measured using the Python \verb,time.clock(), function) based on experiments performed
on a 3.40 GHz Intel Xeon machine, running Ubuntu 16.04.

\subsection{Boolean least squares}\label{s-exp-bool-ls}
Here, we consider the Boolean least squares problem~\eqref{eq-boolean-ls} from~\S\ref{s-applications}.

\paragraph{Problem instance.}
We generated a random problem instance with $n = 50$ and $m = 80$,
where the entries of $A$ and $b$ are drawn IID from $\mathcal{N}(0, 1)$.
The feasible set has $2^n \approx 10^{15}$ points.

\paragraph{Results.}
We tested various combinations of \emph{Suggest} and \emph{Improve} methods.
For each combination,
we sampled 20 candidate points and improved them, and took the best point.
In this example, we considered all three different \emph{Suggest} methods in~\S\ref{s-constructing}: random, spectral, and SDR.
For the \emph{Improve} methods, we considered three methods as follows:
\begin{itemize}
	\item Round: rounding the candidate point to the nearest feasible point, as discussed in~\S\ref{s-special-methods}.
	\item CD: two-phase coordinate descent, as discussed in~\S\ref{s-greedy}.
	\item CCP: penalty CCP, as discussed in~\S\ref{s-dccp}.
\end{itemize}
While it is possible to apply multiple \emph{Improve} methods sequentially,
we did not consider them in this experiment.
We excluded the two-phase ADMM
because it easily fails to generate a feasible point
when applied to Boolean problems,
as noted in page~\pageref{admm-loop}.
For the penalty CCP heuristic, we used the penalty parameter of $\tau = 1$.

All three \emph{Improve} methods yielded a feasible point for every candidate point generated by \emph{Suggest} methods.
In Table~\ref{t-result-boolean-ls}, we show, for every combination of the \emph{Suggest} and \emph{Improve} methods,
the objective value of the best point found and the total running time
(which includes the solve time of the relaxations).

Spectral and semidefinite relaxations yielded lower bounds of $228$ and $518$, respectively.
The best feasible point found had the objective value of $988$, and it was obtained by performing coordinate descent on SDR-based candidate points.
Since the problem instance was small enough to find the global solution,
we used Gurobi~\cite{gurobi} to find the optimal value $f^\star$, which was $920$.
As pointed out in~\cite{park2017semidefinite,d2003relaxations},
this optimal value is closer to the best upper bound of $988$, rather than to the SDR bound of $518$.

Rounding a random point to the nearest feasible point is equivalent to choosing a random feasible point.
As expected, this method performed the worst: the best objective value attained was $2719$.
Significantly better points were obtained by changing the \emph{Suggest} method:
rounding the solution of the spectral relaxation gave a point with objective $1605$, and rounding the SDR-based candidate points gave the best objective of $1098$.
The performance was further improved by performing the two-phase coordinate descent \emph{Improve} method.
With two-phase coordinate descent, even random candidate points yielded a better point than rounding SDR-based candidate points.
This result is quite intuitive, because coordinate descent is a
natural heuristic choice for Boolean problems.
In fact, applying the two-phase coordinate descent in this setting reduces to
a heuristic that performs well in practice~\cite{croes1958method},
which is to round the candidate point to the nearest point (phase I),
and perform a 1-opt local search (phase II).
In this example, the penalty CCP found the same feasible point for every given candidate point,
regardless of the \emph{Suggest} method.

In terms of the running time, we found, for this problem, that there is no benefit in running
the penalty CCP or two-phase ADMM;
the two-phase coordinate descent yielded the best point faster than them.
We note that the running time of Gurobi for finding the optimal value was $1793$ seconds (CPU time).

\begin{table}
\begin{center}
\begin{tabular}{|c||c|c|c|}
\hline
\bf{Best point} & Random & Spectral & SDR  \\ \hline\hline
Round & 2719 & 1605 & 1098 \\ \hline
CD    & 1043 & 1017 &  988 \\ \hline
CCP   & 1063 & 1063 & 1063 \\ \hline
\end{tabular}
\quad
\begin{tabular}{|c||c|c|c|}
\hline
\bf{Runtime} & Random & Spectral & SDR \\ \hline\hline
Round &   0.0 &  26.2 &  41.3 \\ \hline
CD    &  67.3 &  93.8 & 199.9 \\ \hline
CCP   & 872.7 & 877.9 & 968.5 \\ \hline
\end{tabular}
\end{center}
\caption{Objective value of the best point found (left), and total running time (right) of the \emph{Suggest-and-Improve} heuristic on the Boolean least squares problem.}
\label{t-result-boolean-ls}
\end{table}

\subsection{Secondary user multicast beamforming}
We consider the secondary user multicast beamforming problem,
which is a variant of~\eqref{eq-mimo} studied in~\cite{phan2009spectrum,huang2016consensus}.
The problem can be formulated as the following:
\BEQ\label{eq-mimo-secondary}
\begin{array}{ll}
\mbox{minimize} & \|w\|_2^2 \\
\mbox{subject to} & |h_i^H w|^2 \ge \tau, \quad i=1, \ldots, m\\
& |g_j^H w|^2 \le \eta, \quad j=1, \ldots, l.
\end{array}
\EEQ
Here, $w \in \complex^n$ is the variable,
and $h_i, g_j \in \complex^n$ are given problem data,
for $i = 1, \ldots, m$ and $j = 1, \ldots, l$.
Since \texttt{CVXPY} and \texttt{QCQP} do not directly
support complex-valued variables and data, we form an equivalent problem with real numbers only. For $i = 1, \ldots, m$, let
\[
a_i = (\Re h_i, \Im h_i), \quad
b_i = (-\Im h_i, \Re h_i),
\]
be real-valued vectors in $\reals^{2n}$. (Here, $\Re z \in \reals^n$ and $\Im z \in \reals^n$ denote the real and imaginary parts of a complex-valued vector $z \in \complex^n$.)
Similarly, for $j = 1, \ldots, l$, define
\[
c_j = (\Re g_j, \Im g_j), \quad
d_j = (-\Im g_j, \Re g_j).
\]
By introducing another variable $x \in \reals^{2n}$ to represent the real and imaginary parts of $w$, we can rewrite~\eqref{eq-mimo-secondary} using real-valued variables and data only:
\[
\begin{array}{ll}
\mbox{minimize} & \|x\|_2^2 \\
\mbox{subject to}
& (a_i^T x)^2 + (b_i^T x)^2 \ge \tau, \quad i=1, \ldots, m\\
& (c_j^T x)^2 + (d_j^T x)^2 \le \eta, \quad j=1, \ldots, l.
\end{array}
\]

\paragraph{Problem instance.}
We considered a problem instance with $n = 50$, $m = 20$, and $l = 5$.
The entries of the $a_i$, $b_i$, $c_j$, and $d_j$ were drawn IID from $\mathcal{N}(0, 1)$.
The other parameters were set as $\tau = 20$ and $\eta = 2$.

\paragraph{Results.}
As in~\S\ref{s-exp-bool-ls}, we used a combination of \emph{Suggest} and \emph{Improve} methods.
The list of \emph{Improve} method we used for~\eqref{eq-mimo-secondary} is the following:
\begin{itemize}
  \item Scale: scaling the candidate point so that it satisfies the first set of constraints,
  \[
  (a_i^T x)^2 + (b_i^T x)^2 \ge \tau, \quad i=1, \ldots, m,
  \]
  as discussed in~\S\ref{s-special-methods}.
  Unlike~\eqref{eq-mimo}, this problem has a second set of constraints, and therefore, this method does not necessarily yield a feasible point.
	\item CD
	\item ADMM
	\item ADMM/CD
	\item CD/ADMM
	\item CCP
\end{itemize}
For the two-phase ADMM, we chose the penalty parameter $\rho = \sqrt{m+l}$, as in~\cite{huang2016consensus}.
For the other heuristics, we used the default parameters.
For each choice of \emph{Suggest} and \emph{Improve} methods,
we sampled 10 candidate points and improved them, and took the best point.

In Table~\ref{t-result-mimo}, we show, for every combination of the \emph{Suggest} and \emph{Improve} methods,
the objective value of the best feasible point found and the total running time
(which includes the solve time of the relaxations).
The Scale \emph{Improve} method never produced a feasible point for all three \emph{Suggest} methods,
and therefore we left it out from the table.

Spectral and semidefinite relaxations yielded lower bounds of $1.11$ and $1.27$, respectively.
The best feasible point found had the objective value of $1.30$, obtained by performing the penalty CCP on SDR-based candidate points.

Unlike in~\S\ref{s-exp-bool-ls}, the two-phase coordinate descent did not produce a good feasible point.
On random candidate points, the best feasible point it found had the objective value of $9.62$.
Even with the SDR-based candidate points, the best objective value was $2.51$.
The two-phase ADMM, on the other hand, produced a point with objective value $1.86$ when combined with the SDR \emph{Suggest} method.
The penalty CCP method showed a more consistent performance for both random
and SDR \emph{Suggest} methods, and it found the best point with objective $1.30$.
However, it was unable to produce a feasible point from the solution of the spectral relaxation.

We note that running ADMM followed by coordinate descent performed better than applying only one of the heuristics.
However, running the two heuristics in the other order did not find a good feasible point.
This implies that depending on the problem,
the individual \emph{Improve} methods can be considered as building blocks
for more sophisticated \emph{Improve} sequences that can produce a better point.

The running time of the ADMM and penalty CCP \emph{Improve} methods was heavily affected when they were unable to find feasible points.
This issue can be addressed by specifying the maximum number of iterations performed by the methods,
or prematurely terminating the methods when they reach a preset time limit.

\begin{table}
\begin{center}
\begin{tabular}{|c||c|c|c|}
\hline
\bf{Best point} & Random & Spectral & SDR  \\ \hline\hline
CD      &  9.62 & 9.90 & 2.51 \\ \hline
ADMM    &  7.84 & 1.98 & 1.86 \\ \hline
ADMM/CD &  4.64 & 1.73 & 1.42 \\ \hline
CD/ADMM & 11.18 & 8.70 & 3.07 \\ \hline
CCP     &  1.31 & N/A  & 1.30 \\ \hline
\end{tabular}
\quad
\begin{tabular}{|c||c|c|c|}
\hline
\bf{Runtime} & Random & Spectral & SDR \\ \hline\hline
CD      & 195.3 & 572.1 & 790.7 \\ \hline
ADMM    &  1843 &  1237 &  8355 \\ \hline
ADMM/CD &  1636 &  1286 &  6241 \\ \hline
CD/ADMM &  1518 &  1317 &  2173 \\ \hline
CCP     &  5626 &  4770 &  2216 \\ \hline
\end{tabular}
\end{center}
\caption{Objective value of the best point found (left), and total running time (right) of the \emph{Suggest-and-Improve} heuristic on the secondary user multicast beamforming problem.}
\label{t-result-mimo}
\end{table}

\section*{Acknowledgement}
Parts of the paper are based on the class notes of EE392o, originally written by Alexandre d'Aspremont and Stephen Boyd~\cite{d2003relaxations}.

\clearpage
\appendix

\section{Solving QCQP with one variable}\label{s-onevar}
In this section, we consider the following one-variable QCQP:
\BEQ\label{eq-onevar}
\begin{array}{ll}
\mbox{minimize} & p_0 x^2 + q_0 x + r_0 \\
\mbox{subject to} & p_i x^2 + q_i x + r_i \le 0, \quad i=1, \ldots, m,
\end{array}
\EEQ
with variable $x \in \reals$.
We show an $O(m \log m)$ time solution method for solving~\eqref{eq-onevar} using only elementary algebra and data structures.
Consider the $i$th constraint
\[
p_i x^2 + q_i x + r_i \le 0,
\]
and the set $S_i$ of $x$ satisfying the constraint.
Assuming that $S_i$ is nontrivial, \ie, $S_i$ is not equal to $\reals$ or $\emptyset$,
it can take three different forms, depending on the sign of $p_i$:
\begin{itemize}
\item If $p_i = 0$, then the constraint is affine, and thus $S_i$ is a left- or right-unbounded interval. That is, $S_i$ has the form $(-\infty, a]$ or $[a, +\infty)$.
\item If $p_i > 0$, then $S_i$ is a closed interval $[a, b]$.
\item If $p_i < 0$, then $S_i$ is the union of a left-unbounded and a right-unbounded intervals, \ie, $S_i = (-\infty, a] \cup [b, +\infty)$, for some $a$ and $b$ (with $a < b$).
\end{itemize}
We inductively argue that the feasible set $S = S_1 \cap \cdots \cap S_m$ is a collection of
at most $m+1$ disjoint, potentially left- or right-unbounded intervals on $\reals$.
The base case where $m = 0$ gives $S = \reals$, which satisfies the claim.
Now, assume that $S$ is a collection of at most $m+1$ disjoint intervals on $\reals$.
We claim that $S \cap S_{m+1}$ is a collection of at most $m+2$ disjoint intervals on $\reals$.
This holds trivially if $S_{m+1}$ is a convex interval.
The only case where additional intervals can be introduced is when $p_{m+1} < 0$.
Let $S_{m+1} = (-\infty, a] \cup [b, +\infty)$, or equivalently, $S_{m+1} = \reals \setminus (a, b)$.
Taking the intersection of $S$ and $S_{m+1}$ is the same as ``subtracting'' an open interval $U = (a, b)$ from every interval $C$ in $S$.
Subtracting an open interval $U$ from a closed interval $C$ results in another closed interval $C'$,
unless $U \subset C$, in which case $C \setminus U$ is two disjoint closed intervals.
Since the intervals in $S$ are disjoint, there can be at most one $C$ such that $U \subset C$ holds.
This completes the inductive step, \ie, $S \cap S_{m+1}$ consists of at most $m+2$ disjoint intervals.


To compute these disjoint intervals, we use a balanced binary search tree, where each node corresponds to an interval, ordered by the left endpoint.
When taking the intersection of the current intervals with $S_i$, at most two intervals can change their endpoints, and some intervals can be removed from the data structure.
The case where $p_i < 0$ is the only case we potentially insert an additional interval.

Every insertion, deletion, or modification (which can be thought of as deletion followed by insertion) takes $O(\log N)$ time, where $N$ is the maximum number of nodes in the binary search tree at any given time.
To analyze the overall time complexity,
we can simply count the number of each operation types.
As argued above, the total number of insertions is bounded by $m$.
The number of deletions is naturally bounded by the number of insertions, which is also $m$.
The number of modifications is bounded by $2m$.
It follows that the time complexity of computing the feasible set is $O(m \log m)$.

Minimizing a quadratic function in this feasible set $\mathcal{S}$ can be done by evaluating
the objective at the endpoints of the intervals, as well as checking the
unconstrained minimizer (only when $p_0 > 0$).

\section{Solving QCQP with one constraint}\label{s-oneqcqp}

In this section, we describe the solution method for the following special case of~\eqref{eq-oneqcqp-general}:
\BEQ\label{eq-oneqcqp}
\begin{array}{ll}
\mbox{minimize} & \|x - z\|_2^2 \\
\mbox{subject to} & x^T P x + q^T x + r \le 0,
\end{array}
\EEQ
with variable $x \in \reals^n$.
For simplicity, we assume that the constraint is satisfiable for some $x \in \reals^n$.
Satisfiability is easily verified from the eigenvalue decomposition of $P$.

If $z$ satisfies the constraint, \ie, $z^T P z + q^T z + r \le 0$, then it is clear that $x^\star = z$ is the solution and there is nothing else to do.
Otherwise, due to complementary slackness~\cite[\S5.5.2]{boyd2004convex}, any optimal point $x^\star$ must satisfy the constraint with equality:
\[
x^{\star T} P x^\star + q^T x^\star + r = 0.
\]
Then, without loss of generality, we can work with the equality constrained version of~\eqref{eq-oneqcqp}:
\BEQ\label{eq-oneqcqp-eq}
\begin{array}{ll}
\mbox{minimize} & \|x - z\|_2^2 \\
\mbox{subject to} & x^T P x + q^T x + r = 0.
\end{array}
\EEQ
Let $P = Q\Lambda Q^T$ be the eigenvalue decomposition of $P$, with $\Lambda = \diag(\lambda_1, \ldots, \lambda_n)$. Consider the following problem, equivalent to~\eqref{eq-oneqcqp-eq}:
\BEQ\label{eq-oneqcqp-tr}
\begin{array}{ll}
\mbox{minimize} & \|\hat{x} - \hat{z}\|_2^2 \\
\mbox{subject to} & \hat{x}^T \Lambda \hat{x} + \hat{q}^T \hat{x} + r = 0,
\end{array}
\EEQ
where $\hat{z} = Q^T z$, $\hat{q} = Q^T q$, and $\hat{x} = Q^T x$.

The Lagrangian of~\eqref{eq-oneqcqp-tr} is given by:
\[
\mathcal{L}(\hat{x}, \nu) = \hat{x}^T (I+\nu\Lambda) \hat{x} + (\nu \hat{q} - 2\hat{z})^T \hat{x} + \nu c + \|\hat{z}\|_2^2.
\]
Since we assumed the feasibility of~\eqref{eq-oneqcqp-tr}, there must exist $\nu$ with $I+\nu \Lambda \succeq 0$,
such that the value of $\hat{x}$ minimizing the
Lagrangian also satisfies the equality
constraint of~\eqref{eq-oneqcqp-tr}.

There are two cases to consider:
\begin{itemize}
\item Case (i): $I+\nu\Lambda \succ 0$. \\
First, consider the range of $\nu$ so that $I + \nu \Lambda \succ 0$, or equivalently, $1 + \nu \lambda_i > 0$ for all $i$.
In this range, we can find the value of $\hat{x}$ minimizing the Lagrangian
by taking the derivative with respect to $\hat{x}$ and setting it to zero:
\BEQ\label{eq-xhat}
\hat{x} = -(1/2)(I + \nu\Lambda)^{-1}(\nu \hat{q} - 2\hat{z}).
\EEQ
By substituting $\hat{x}$ into the equality constraint and expanding out, we get a nonlinear equation in $\nu$:
\BEQ\label{eq-secular}
\sum_{i=1}^n \left( \frac{\lambda_i (\nu \hat{q}_i - 2\hat{z}_i)^2}{4(1+\nu\lambda_i)^2}  - \frac{\hat{q}_i (\nu \hat{q}_i - 2\hat{z}_i)}{2(1+\nu\lambda_i)} \right) + r = 0.
\EEQ
If this equation attains a solution in the range where $1+\nu\lambda_i > 0$
for all $i$, then the corresponding $\hat{x}$ given by~\eqref{eq-xhat} is the
solution to~\eqref{eq-oneqcqp-tr}. In order to solve~\eqref{eq-secular}
numerically, observe that the derivative of the lefthand side is
\[
-\sum_{i=1}^n \frac{(2 \lambda_i \hat{z}_i + \hat{q}_i)^2}{2 (1 + \lambda_i \nu)^3} < 0,
\]
and thus the lefthand side monotonically decreases in $\nu$.
Then, one can solve~\eqref{eq-secular} by checking where the sign of the lefthand side changes, using either bisection or Newton's method.

\item Case (ii): $I+\nu\Lambda \succeq 0$ and $I+\nu\Lambda$ is singular.\\
There are at most two values of $\nu$ where this is possible.
If $\lambda_{\min} = \min_i \lambda_i < 0$, then $\nu = -1/\lambda_{\min}$ is one such value.
If $\lambda_{\max} = \max_i \lambda_i > 0$, then $\nu = -1/\lambda_{\max}$ is another such value.
For these values of $\nu$, we can simply check if there exists $\hat{x}$ where the Karush-Kuhn-Tucker (KKT) conditions~\cite[\S5.5.3]{boyd2004convex} are held:
\[
2(I+\nu \Lambda)\hat{x} = 2\hat{z} - \nu \hat{q},
\qquad
\hat{x}^T \Lambda \hat{x} + \hat{q}^T \hat{x} + r = 0.
\]
\end{itemize}
Once $\hat{x}$ is found, then the solution $x^\star$ of~\eqref{eq-oneqcqp} is simply given by $x^\star = Q\hat{x}$.
\section{Splitting quadratic forms}\label{s-split}
In this appendix, we explore various ways of splitting an
indefinite matrix $P$ as a difference $P_+ - P_-$
of two positive semidefinite matrices, as in~\eqref{quadratic-split}.
The motivation for this discussion is that the performance
of convex-concave procedure can vary drastically depending
on how the quadratic form is split into convex and concave
parts~\cite{lipp2016variations,shen2016disciplined}.

\subsection{Desired properties}

In addition to the running time,
there are several other properties of the algorithm that are important in typical applications.

\paragraph{Time complexity.}
Clearly, we want the algorithm for computing the representation to be fast.
Asymptotic performance is important, but in practice,
even with the same asymptotic complexity,
one method may outperform another depending on how
sparsity is exploited.

\paragraph{Memory usage.}
In the worst case where both $P_+$ and $P_-$ are dense,
$2n^2$ floating point numbers must be stored.
If $P$ has a special sparsity pattern that can be exploited,
it is often desirable to keep the pattern as much as possible,
and represent $P_+$ and $P_-$ without introducing too many nonzeros.
Ideally, representing $P_+$ and $P_-$ should not cost much more than
representing $P$ itself.

\paragraph{Numerical stability.}
Due to the round-off errors of floating point numbers, the resulting
representation may be numerically inaccurate. This can be particularly
problematic when $P_+$ and $P_-$ are represented implicitly
(\eg, via Cholesky factorization),
rather than explicitly. Singular or badly conditioned matrices can often
introduce big round-off errors, and we want the algorithm to be robust in such
settings.

\paragraph{Additional curvature.}
We want $P_+$ and
$P_-$ to introduce as little ``distortion'' as possible,
\ie, we want the additional curvature measured by the nuclear norm
\[
\|P_+\|_* + \|P_-\|_* - \|P\|_*
\]
to be small. Since $P_+$ and $P_-$ are required to be positive semidefinite, this quantity is equal to
\[
\Tr P_+ + \Tr P_- - \|P\|_*.
\]

\subsection{Simple representations}

Any indefinite matrix can be made positive semidefinite by adding a large
enough multiple of the identity.
Let $\lambda_{\min} < 0$ be the smallest eigenvalue of $P$. Then, for any
$t \ge |\lambda_{\min}|$,
\[
P_+ = P + t I, \qquad P_- = t I,
\]
is a pair of positive semidefinite matrices whose difference is $P$.
If the magnitude of the maximum eigenvalue $\lambda_{\max} > 0$ is smaller than that of $\lambda_{\min}$, then an alternative representation is also possible:
\[
P_+ = tI, \qquad P_- = tI - P,
\]
where $t \ge \lambda_{\max}$. This representation is relatively easy to compute as it only requires a lower bound
on $|\lambda_{\min}|$ or $|\lambda_{\max}|$.
It also has a property that the sparsity of $P$ is
preserved as much as possible, in that no new off-diagonal nonzero entries are
introduced in $P_+$ or $P_-$. Its disadvantage is the additional curvature it introduces when $t$ is large:
\[
\|P_+\|_* + \|P_-\|_* - \|P\|_* = 2t.
\]

Another simple representation is based on the full eigenvalue decomposition of
$P$. This representation preserves the norm of $P$ and thus introduces no
additional curvature. Let $P = Q\Lambda Q^T$ be the eigenvalue decomposition
of $P$, where $\Lambda = \diag(\lambda_1, \ldots, \lambda_n)$ is the
eigenvalue matrix with
\[
\lambda_1 \ge \cdots \ge \lambda_k \ge 0 > \lambda_{k+1} \ge \cdots \ge \lambda_n.
\]
Then, $\Lambda$ can be
written as $\Lambda = \Lambda_+ - \Lambda_-$, where
\[
\Lambda_+ = \diag(\lambda_1, \ldots, \lambda_k, 0, \ldots, 0), \qquad
\Lambda_- = \diag(0, \ldots, 0, -\lambda_{k+1}, \ldots, -\lambda_n).
\]
Setting
\[
P_+ = Q\Lambda_+ Q^T, \qquad P_+ = Q\Lambda_- Q^T
\]
then gives a difference of positive semidefinite matrix representation of $P$.
This approach has a high computational cost from the full eigenvalue decomposition (which takes $O(n^3)$ time).
Also, in general, even when $P$ is sparse, the resulting matrices $P_+$ and $P_-$ are dense.

\subsection{Cholesky-like representations}\label{s-chol-rep}

When $P$ is positive definite, there exists a unique
lower triangular matrix $L \in \reals^{n \times n}$, called the \emph{Cholesky factor} of $P$, that satisfies
$P = LL^T$. This representation is known as
the \emph{Cholesky factorization} of $P$
(see, \eg,~\cite[\S C.3]{boyd2004convex}).
In this subsection, we explore a representation based on the
Cholesky factorization.
We start by describing the Cholesky algorithm,
which, in the simplest form, is a recursive algorithm.
If $P$ is $1$-by-$1$, then $L$ is
simply given by $\sqrt{P_{11}}$. If $P$ has two or more rows, let
\[
P = \left[ \begin{array}{cc} a & v^T \\ v & M \end{array} \right],
\]
with $a > 0$,
$v \in \reals^{n-1}$, $M \in \symm^{n-1}$.
Then, the Cholesky factor $L$ of $P$ is given by:
\[
L = \left[ \begin{array}{cc} \sqrt{a} & \\ v/\sqrt{a} & L^{\prime} \end{array} \right],
\]
where $L^{\prime}$ is the Cholesky factor of $M-(1/a)vv^T$.
The cost of computing a dense Cholesky factorization is $(1/3)n^3$ flops. In
case $P$ is sparse, various pivoting heuristics can be used to exploit the
sparsity structure and speed up the computation~\cite[\S6.3]{nocedal1999numerical}.

Cholesky factorization does not exist when $P$ is indefinite.
However, the LDL decomposition, which is a close variant of the
Cholesky factorization, exists for all symmetric matrices~\cite{bunch1971direct}.
It also has an
additional computational advantage since there is no need to take square
roots.
The idea of the LDL decomposition is to write $P$ as $LDL^T$, where $L \in \reals^{n \times n}$
is lower triangular with ones on the main diagonal, and
$D \in \reals^{n \times n}$ is block diagonal consisting of $1$-by-$1$ or
$2$-by-$2$ blocks.
When $D$
is diagonal (\ie, no $2$-by-$2$ blocks), then one can easily separate out the
negative entries of $D$ and the corresponding columns in $L$ to write $P$ as
\[
P = L_1 D_1 L_1^T - L_2 D_2 L_2^T,
\]
where $D_1$ and $D_2$ are nonnegative diagonal matrices.
When $D$ has $2$-by-$2$ blocks, however, it is not possible
to directly transform the LDL decomposition into the desired form.

\subsection{Difference-of-Cholesky representation}

In this subsection, we develop an adaptation
of the Cholesky decomposition that returns
a ``difference-of-Cholesky'' representation.
There are several advantages of this approach
over the method discussed in~\S\ref{s-chol-rep}.
First, $P$ can be any indefinite matrix, \ie, its LDL
decomposition can contain $2$-by-$2$ blocks.
Additionally, the method has a parameter $\delta > 0$
that controls the numerical robustness against division by small numbers;
the algorithm will never perform division by numbers
whose magnitude is smaller than $\sqrt{\delta}$,
at the expense of additional curvature.
Finally, various heuristics for Cholesky decomposition that
promote sparsity can be applied directly to this modification.

Let $P \in \symm^n$ be a symmetric matrix of two or more rows:
\[
P = \left[ \begin{array}{cc} a & v^T \\ v & M \end{array} \right],
\]
where $a \in \reals$, $v \in \reals^{n-1}$, $M \in \symm^{n-1}$.
Our goal is to find a pair of
lower triangular matrices $L_1$ and $L_2$ such that
\[
P = L_1 L_1^T - L_2 L_2^T.
\]
As in~\S\ref{s-chol-rep}, we show a recursion for $L_1$ and $L_2$.
Without loss of generality, we assume that $a \ge 0$; if $a < 0$, we can simply find the difference-of-Cholesky representation for $-P$ and swap $L_1$ and $L_2$.
Then, there are only two cases to consider, depending on the magnitude of $a$.
\begin{itemize}
\item Case (i): $a > \delta$.\\
In this case, the recursion is given by:
\[
L_1 = \left[ \begin{array}{cc} \sqrt{a} & \\ v/\sqrt{a} & L'_1 \end{array} \right], \quad
L_2 = \left[ \begin{array}{cc} 0 & \\ 0 & L'_2 \end{array} \right].
\]
Here, $L'_1$ and $L'_2$ are lower triangular matrices satisfying
\[
L'_1 L^{\prime T}_1 - L'_2 L^{\prime T}_2 = M-(1/a) vv^T,
\]
which can be obtained by recursively applying the algorithm.
\item Case (ii): $0 \le a \le \delta$.\\
The recursion in this case has the following form:
\[
L_1 = \left[ \begin{array}{cc} \sqrt{\delta+a} & \\ v_1 & L'_1 \end{array} \right], \quad
L_2 = \left[ \begin{array}{cc} \sqrt{\delta} & \\ v_2 & L'_2 \end{array} \right].
\]
Here, $v_1$ and $v_2$ are arbitrary vectors satisfying
\[
\sqrt{\delta + a} \, v_1 - \sqrt{\delta} \, v_2 = v,
\]
and $L'_1$ and $L'_2$ are lower triangular matrices with
\[
L'_1 L^{\prime T}_1 - L'_2 L^{\prime T}_2 = M-v_1 v_1^T + v_2 v_2^T.
\]
Note that we have an additional degree of freedom,
as we can freely choose $v_1$ or $v_2$.
For example, by letting $v_1 = 0$ or $v_2 = 0$,
it is possible to trade off the number
of nonzero elements between $L_1$ and $L_2$, which is a property that was not
readily available in the other representations we discussed.
When we choose $v_1 = 0$, the recursion becomes
\[
L_1 = \left[ \begin{array}{cc} \sqrt{\delta+a} & \\ 0 & L'_1 \end{array} \right], \quad
L_2 = \left[ \begin{array}{cc} \sqrt{\delta} & \\ -v/\sqrt{\delta} & L'_2 \end{array} \right],
\]
with
\[
L'_1 L^{\prime T}_1 - L'_2 L^{\prime T}_2 = M+(1/\delta) vv^T.
\]
Similarly, if we choose $v_2 = 0$, we get
\[
L_1 = \left[ \begin{array}{cc} \sqrt{\delta+a} & \\ v/\sqrt{\delta+a} & L'_1 \end{array} \right], \quad
L_2 = \left[ \begin{array}{cc} \sqrt{\delta} & \\ 0 & L'_2 \end{array} \right],
\]
with
\[
L'_1 L^{\prime T}_1 - L'_2 L^{\prime T}_2 = M-\frac{1}{\delta+a} \, vv^T.
\]
\end{itemize}

This method is very close to the original Cholesky algorithm, and
most extensions and techniques applicable to Cholesky factorization
naturally apply to the difference-of-Cholesky method.
For example, it is simple to modify this algorithm to return a pair of LDL factorizations
to avoid computing square roots.
Other examples include pivoting heuristics
for avoiding round-off
errors or ill-conditioned matrices.

\section{Source examples}\label{s-source-codes}

In this section, we give the full version of the script used in~\S\ref{s-usage},
as well as sample implementations of the \emph{Suggest-and-Improve} framework
for the numerical examples considered in~\S\ref{s-experiment}.
The codes are written using our open source package \texttt{QCQP}.
We put special emphasis on the similarities between the description of Algorithm~\ref{master-algo} and the actual codes,
as well as the short length of the codes.

\subsection{Sample usage}
\begin{verbatim}
import numpy as np, cvxpy as cvx
from qcqp import *

np.random.seed(1)

# Problem data
n = 10
W0 = np.random.randn(n, n)
W = 0.5*(W0 + W0.T)

# Construct a nonconvex QCQP
x = cvx.Variable(n)
prob = cvx.Problem(
    cvx.Maximize(cvx.quad_form(x, W)),
    [cvx.square(x) == 1]
)
qcqp = QCQP(prob)

# Solve the spectral relaxation
qcqp.suggest(SPECTRAL, solver=cvx.MOSEK)

# Print bounds
print ("Spectral bound: %.4f" % qcqp.spectral_bound)
(w, v) = np.linalg.eig(W)
print ("n*lambda_max:   %.4f" % (max(w)*n))

# Print solution of the relaxation
print(x.value)
x_rl = np.sqrt(n) * v[:, np.argmax(w)]
print(x_rl)

# Use coordinate descent heuristic
f_cd, v_cd = qcqp.improve(COORD_DESCENT)
print ("Objective: %.4f" % f_cd)
print ("Maximum violation: %.4f" % v_cd)
print(x.value)

# Find the optimal solution by enumeration
f_star = -np.inf
mask_star = None
for mask in range(2**n):
    x = np.array([2*((mask>>i)&1)-1 for i in range(n)])
    f = np.dot(W.dot(x), x)
    if f_star < f:
        f_star = f
        mask_star = mask
print(f_star)

x_star = [2 * ((mask_star>>i)&1) - 1 for i in range(n)]
print(x_star)
\end{verbatim}

\subsection{Boolean least squares}
\begin{verbatim}
import numpy as np, cvxpy as cvx
from qcqp import *

# Problem data
n, m = 10, 15
A = np.random.randn(m, n)
b = np.random.randn(m, 1)

# Construct a nonconvex QCQP
x = cvx.Variable(n)
prob = cvx.Problem(
    cvx.Minimize(cvx.sum_squares(A*x - b)),
    [cvx.square(x) == 1]
)
qcqp = QCQP(prob)

# Solve the SDR and get a starting point to a local method
qcqp.suggest(SDR, solver=cvx.MOSEK)
print (qcqp.sdr_bound)

# Attempt to improve the starting point using coordinate descent
f, v = qcqp.improve(COORD_DESCENT)
print (f, v) # objective value, maximum violation
print (x.value) # value of the variable
\end{verbatim}

\subsection{Secondary user multicast beamforming}
\begin{verbatim}
import numpy as np, cvxpy as cvx
from cvxpy import *
from qcqp import *

# Problem data
n, m, l = 30, 10, 3
tau, eta = 20, 2
HR, HI = np.random.randn(m, n), np.random.randn(m, n)
GR, GI = np.random.randn(l, n), np.random.randn(l, n)
A, B = np.hstack((HR, HI)), np.hstack((-HI, HR))
C, D = np.hstack((GR, GI)), np.hstack((-GI, GR))

# Construct a nonconvex QCQP
x = cvx.Variable(2*n)
prob = cvx.Problem(
    cvx.Minimize(cvx.sum_squares(x)),
    [
        cvx.square(A*x) + cvx.square(B*x) >= tau,
        cvx.square(C*x) + cvx.square(D*x) <= eta
    ]
)
qcqp = QCQP(prob)

# Solve the SDR and get a starting point to a local method
qcqp.suggest(SDR, solver=cvx.MOSEK)
print (qcqp.sdr_bound)

# Attempt to improve the starting point using ADMM
f, v = qcqp.improve(ADMM, rho=np.sqrt(m+l))
print (f, v) # objective value, maximum violation
print (x.value) # value of the variable
\end{verbatim}

\clearpage
\nocite{*}\bibliography{refs}
\end{document}